\documentclass[12pt]{amsart}
\usepackage{amsmath, amsfonts, amsbsy, amsthm, amscd, graphicx}
\usepackage{amssymb, latexsym, color}

\numberwithin{equation}{section}

\theoremstyle{plain}
\newtheorem{Theorem}{Theorem}[section]
\newtheorem{Proposition}[Theorem]{Proposition}
\newtheorem{Lemma}[Theorem]{Lemma}

\theoremstyle{definition}
\newtheorem{Definition}[Theorem]{Definition}

\theoremstyle{remark}
\newtheorem{Remark}{Remark}
\newtheorem{Example}{Example}

\begin{document}

\title{Quaternionic CR Geometry}

\author{Hiroyuki Kamada}
\thanks{Hiroyuki Kamada, Miyagi University of Education,
149 Aramaki-Aoba, Aoba-ku, Sendai 980-0845, Japan\\ 
{\em E-mail address}: {\tt hkamada@staff.miyakyo-u.ac.jp}\\
Partly supported by 
the Grant-in-Aid for Scientific Research (C), 
Japan Society for the Promotion of Science.}

\author{Shin Nayatani}
\thanks{Shin Nayatani, Graduate School of Mathematics, Nagoya University,
Chikusa-ku, Nagoya 464-8602, Japan\\
{\em E-mail address}: {\tt nayatani@math.nagoya-u.ac.jp}\\
Partly supported by the Grant-in-Aid
for Scientific Research (B), 
Japan Society for the Promotion of Science.}



\dedicatory{Dedicated to Professor Seiki Nishikawa 
on his sixtieth birthday}

\maketitle

\newcommand{\qhyp}{\mathop{H^{n+1}_{\Bbb H}}\nolimits}
\newcommand{\smallqhyp}{\mathop{H^{k+1}_{\Bbb H}}\nolimits}
\newcommand{\chyp}{\mathop{H^{n+1}_{\Bbb C}}\nolimits}
\newcommand{\smallchyp}{\mathop{H^{k+1}_{\Bbb C}}\nolimits}
\newcommand{\rhyp}{\mathop{H^{n+1}_{\Bbb R}}\nolimits}
\newcommand{\smallrhyp}{\mathop{H^{k+1}_{\Bbb R}}\nolimits}
\newcommand{\fhyp}{\mathop{H^{n+1}_{\Bbb F}}\nolimits}
\newcommand{\qball}{\mathop{B^{n+1}_{\Bbb H}}\nolimits}
\newcommand{\smallqball}{\mathop{B^{k+1}_{\Bbb H}}\nolimits}
\newcommand{\cball}{\mathop{B^{n+1}_{\Bbb C}}\nolimits}
\newcommand{\smallcball}{\mathop{B^{k+1}_{\Bbb C}}\nolimits}
\newcommand{\rball}{\mathop{B^{n+1}_{\Bbb R}}\nolimits}
\newcommand{\smallrball}{\mathop{B^{k+1}_{\Bbb R}}\nolimits}
\newcommand{\fball}{\mathop{B^{n+1}_{\Bbb F}}\nolimits}
\newcommand{\qsphere}{\mathop{S^{4n+3}}\nolimits}
\newcommand{\smallqsphere}{\mathop{S^{4k+3}_{\Bbb H}}\nolimits}
\newcommand{\chain}{\mathop{S_{\Bbb C}^1}\nolimits}
\newcommand{\rcircle}{\mathop{S_{\Bbb R}^1}\nolimits}
\newcommand{\smallcsphere}{\mathop{S^{2k+1}_{\Bbb C}}\nolimits}
\newcommand{\smallrsphere}{\mathop{S^k_{\Bbb R}}\nolimits}
\newcommand{\qhgroup}{\mathop{{\mathcal{H}}^{4n+3}}\nolimits}
\newcommand{\chgroup}{\mathop{{\mathcal{H}}^{2n+1}_{\Bbb C}}\nolimits}
\newcommand{\binfty}{\mathop{\partial_\infty}\nolimits}
\newcommand{\limitset}{\mathop{\Lambda(\Gamma)}\nolimits}
\newcommand{\domain}{\mathop{\Omega(\Gamma)}\nolimits}
\newcommand{\critical}{\mathop{\delta(\Gamma)}\nolimits}
\newcommand{\crauto}{\mathop{{\rm Aut}_{CR}(S^{2n+1})}\nolimits}
\newcommand{\qisom}{\mathop{G_{\Bbb H}(n+1)}\nolimits}
\newcommand{\cisom}{\mathop{G_{\Bbb C}(n+1)}\nolimits}
\newcommand{\smallcisom}{\mathop{G_{\Bbb C}(k+1)}\nolimits}
\newcommand{\risom}{\mathop{G_{\Bbb R}(n+1)}\nolimits}
\newcommand{\smallrisom}{\mathop{G_{\Bbb R}(k+1)}\nolimits}
\newcommand{\fisom}{\mathop{G_{\Bbb F}(n+1)}\nolimits}
\newcommand{\qhyperboloid}{\mathop{{\frak H}^{4n+3}}\nolimits}
\newcommand{\chyperboloid}{\mathop{{\frak H}^{2n+1}_{\Bbb C}}\nolimits}

\newcommand{\CC}{Carnot-Carath\'{e}odory\enskip}
\newcommand{\TW}{Tanaka-Webster\enskip}
\newcommand{\qua}{quaternionic\enskip}
\newcommand{\spc}{strongly pseudoconvex\enskip}
\newcommand{\PS}{Patterson-Sullivan\enskip}
\newcommand{\Ric}{\text{\rm Ric}}
\newcommand{\Tor}{\text{\rm Tor}}
\newcommand{\tr}{\text{\rm tr}}
\newcommand{\im}{\text{\rm Im}}
\newcommand{\re}{\text{\rm Re}}
\newcommand{\sym}{\text{\rm sym}}
\newcommand{\Levi}{\text{\rm Levi}}
\newcommand{\rmo}{\mathop{\sqrt{-1}}\nolimits}
\newcommand{\del}{\mathop{\partial}\nolimits}
\newcommand{\delbar}{\mathop{\bar{\partial}}\nolimits}
\newcommand{\onezero}{\mathop{Q^{1,0}}\nolimits}
\newcommand{\zeroone}{\mathop{Q^{0,1}}\nolimits}
\newcommand{\rank}{\mathop{\rm rank}\nolimits}
\newcommand{\id}{\mathop{\rm Id}\nolimits}

\newcommand{\define}{\mathop{\stackrel{\rm def}{\leftrightarrow}}\nolimits}
\newcommand{\defineeq}{\mathop{\stackrel{\rm def}{=}}\nolimits}

\newcommand{\R}{\mathop{\Bbb R}\nolimits}
\newcommand{\C}{\mathop{\Bbb C}\nolimits}
\renewcommand{\H}{\mathop{\Bbb H}\nolimits}

\renewcommand{\mod}{\mathrm{mod}\,\,}

\begin{abstract}
Modelled on a real hypersurface in a quaternionic manifold, we introduce 
a quaternionic analogue of CR structure, called quaternionic CR structure. 
We define the strong pseudoconvexity of this structure as well as 
the notion of quaternionic pseudohermitian structure. 
Following the construction of the Tanaka-Webster connection in complex 
CR geometry, we construct a canonical connection associated with a 
quaternionic pseudohermitian structure, when the underlying quaternionic 
CR structure satisfies the ultra-pseudoconvexity 
which is 
stronger than the strong pseudoconvexity. 
Comparison to Biquard's quaternionic contact structure \cite{biq} 
is also made. 
\end{abstract}

\section*{Introduction.}
A CR structure is a corank one subbundle of the tangent
bundle of an odd dimensional manifold, equipped with a complex
structure.
Such a structure typically arises on a real hypersurface
of a complex manifold.
Assuming that the CR structure is strongly pseudoconvex,
the underlying subbundle defines
a contact structure on the manifold,
and a strongly pseudoconvex CR structure together with a choice of contact form 
is called a pseudohermitian structure.
Associated with a pseudohermitian structure, there is a hermitian metric
on the subbundle, called the Levi form.
It is the simplest and most important example of \CC metric.
In pseudohermitian geometry the so-called \TW connection \cite{tana}, \cite{web}
plays the role of the Levi-Civita connection in Riemannian
geometry.
Multiplying the contact form by a nowhere vanishing function
gives another pseudohermitian structure, and accordingly the Levi form
changes conformally, being multiplied by the same function.
CR geometry thereby has a nature of conformal geometry,
and in particular, CR invariants can be computed as those
pseudohermitian invariants which are independent of the choice of
contact form.

In this paper, we shall introduce \qua analogues of CR and
pseudohermitian structures and lay the foundation of the geometry of
these structures.
We introduce two kinds of \qua analogues of CR structure,
with one refining the other.
An almost hyper CR structure is defined on a manifold of dimension
$4n+3$, as a pair of almost CR structures whose underlying
subbundles are transversal to each other and whose complex
structures are anti-commuting in an appropriate sense.
Then the third almost CR structure can be defined, and 
a corank three subbundle is defined as the intersection of the three corank one subbundles. 
The three complex structures leave this bundle invariant, and satisfy the quaternion
relations there. 
An almost hyper CR structure which satisfies a certain integrability condition 
is called a hyper CR structure. 
Analogously to the complex CR case, any real hypersurface
of a hypercomplex manifold has a natural hyper CR structure 
(satisfying a stronger integrability condition). 
A hyper CR structure exists also on a real hypersurface 
of a quaternionic manifold, but only locally. 
In order to have a global structure on any real hypersurface 
of a quaternionic manifold, we refine the notion of hyper CR structure. 
A \qua CR structure is a covering of a manifold by local
hyper CR structures, satisfying a certain gluing condition on
each domain where two such local structures overlap.

Associated with a hyper CR structure, there is a distinguished 
$\R^3$-valued one-form, unique up to 
multiplication by nowhere-vanishing real-valued functions.
For a choice of such one-form, the corresponding Levi form
is a \qua hermitian form on the corank three subbundle.
The strong pseudoconvexity of a hyper CR structure and the
pseudohermitian structure are then defined exactly as in the complex CR case.
On the other hand, the definition of the Levi form itself
is not quite similar to that in the complex CR case, 
and this is the point where the integrability of the hyper CR structure 
does play the crucial role.
These notions of Levi form, strong pseudoconvexity and pseudohermitian structure 
introduced for the hyper CR structure extend to the \qua CR structure.

With these structures at hand, our main concern is whether
there exists a \qua analogue of the \TW connection,
associated with a quaternionic pseudohermitian structure. 
In the hyper pseudohermitian case, our search for such a canonical
connection proceeds as follows.
There is a distinguished family of three-plane fields transverse to
the corank three subbundle and parametrized by sections of the corank three subbundle.
We choose such a three-plane field, and use it to define a one-parameter
family of Riemannian metrics on the manifold, extending
the Levi form on the corank three subbundle.
We then construct an affine connection, characterized by the
property that it has the smallest torsion among those affine
connections with respect to which the above Riemannian metrics
are all parallel. 

What remains to be done is to determine the transverse three-plane field 
so that the corresponding connection 
be best adapted to the structure under consideration in an appropriate sense. 
When a hyper pseudohermitian structure is given, 
the corank three subbundle associated with the underlying hyper CR structure 
comes equipped with an $Sp(n)$-structure. 
We will, however, be moderate by regarding the bundle as an $Sp(n)\cdot Sp(1)$-bundle, 
for the sake of later generalization of the construction to the quaternionic CR case. 
Then the best possible one can expect is that the connection restricts to an 
$Sp(n)\cdot Sp(1)$-connection on the bundle.
Since this turns out not to be possible in general, we will be 
contented by 
requiring that the connection be ``as close to an $Sp(n)\cdot Sp(1)$-connection as possible." 
The last expression will be made explicit by using the representation theory for $Sp(n)\cdot Sp(1)$. 
Note that for this strategy to work, we must primarily assume that $n\geq 2$, that is, the dimension of the underlying manifold is greater than seven; when $n=1$, 
since $Sp(1)\cdot Sp(1) = SO(4)$, 
any orthogonal connection on the oriented corank three bundle should necessarily be an 
$Sp(1)\cdot Sp(1)$-connection. 
It turns out that we must also assume that the hyper CR structure is what we call 
ultra-pseudoconvex. 
When these assumptions are satisfied, the above strategy completely works and thereby gives 
a connection in search. 
Note that the class of hyper CR structures which are ultra-pseudoconvex contains all 
strictly convex real hypersurfaces in ${\Bbb H}^{n+1}$. 
The notion of ultra-pseudoconvexity and the construction of the canonical connection extend 
to the \qua CR case.
It remains to see whether a canonical connection can be constructed when the dimension 
of the underlying manifold is seven. 
We will address this problem in a future work.

It should be mentioned that several quaternionic analogues of CR structures 
other than those in this paper have been introduced and studied 
by Hernandez \cite{hern}, Biquard \cite{biq}, 
Alekseevsky-Kamishima \cite{alkami0}, \cite{alkami} and others. 
Among them, Biquard's quaternionic contact structure is most 
influential and extensively studied. 
We therefore compare our quaternionic CR structure to the quaternionic 
contact structure. 
We observe that while a quaternionic contact structure can always 
be ``extended'' to a quaternionic CR structure, the quaternionic contact structure is 
more restrictive than the quaternionic CR structure. 
Indeed, we characterize, in terms of the Levi forms, a quaternionic CR 
structure whose underlying 
corank three subbundle has compatible quaternionic contact structure. 
We also give explicit examples of quaternionic CR
manifolds which do not satisfy the characterizing condition. 

More recently, Duchemin \cite{duc2} introduced the notion of 
weakly quaternionic contact structure, generalizing that of quaternionic 
contact structure. 
He showed that a real hypersurface in a quaternionic manifold admitted a canonical 
weakly quaternionic contact structure. 
More generally, one easily verifies that a quaternionic CR structure naturally 
produces a weakly quaternionic contact structure. 
As mentioned in \cite[\S 6]{duc2}, the construction of a canonical connection for a 
weakly quaternionic contact structure, 
generalizing the so-called Biquard connection for a quaternionic contact structure, 
remains to be done.

This paper is organized as follows: 
In \S 1, we introduce the definitions of hyper and quaternionic CR 
structures as well as hyper and quaternionic pseudohermitian sturctures.
In \S 2, we give examples of hyper and quaternionic CR manifolds. 
In \S 3, we construct a canonical connection associated with a hyper/quaternionic 
pseudohermitian structure, when the underlying hyper/quaternionic  
CR structure is ultra-pseudoconvex. 
The proof of a technical lemma is postponed to \S 4. 
Comparison to Biquard's quaternionic contact structure is made in \S 5. 
In Appendix, we give proofs of some fundamental facts which are stated in \S 1. 

The main contents of an earlier version of this paper were announced 
in \cite{kana}. However, some significant changes have been made in the 
present manuscript. Among others, we modified the definition of the 
integrability of hyper CR structure.  
The former definition required for each of the three CR structures 
constituting a hyper CR structure to be integrable as a CR structure. 
As pointed out by the referee, this definition had 
the following demerit: if one has a candidate for a quaternionic CR 
structure, e.g., a real hypersurface in a quaternionic manifold, 
one cannot tell whether there are integrable choices of local hyper CR 
structures inducing it, unless the quaternionic manifold is e.g., a hypercomplex manifold. 
As mentioned above, under the new definition of integrability, 
any real hypersurface in a quaternionic manifold has a natural 
quaternionic CR structure. 

Some words on notation. Throughout this paper, the triple of indices $(a,b,c)$ 
always stands for a cyclic permutation of $(1,2,3)$, unless otherwise stated.

\medskip\noindent
{\bf Acknowledgements.}\quad 
The first author thanks Andrew Swann and Martin Svensson for hospitality, suggestions 
and encouragement while he was visiting the University of Southern Denmark, Odense. 
The second author thanks G\'erard Besson for hospitality while visiting the University of Grenoble, 
where a part of this work was done. 
Both the authors thank the referee for his/her critical comments, 
which were crucial in improving the manuscript. 

\section{Hyper \& \qua CR structures and strong pseudoconvexity}
We start with a brief review of CR structure.
Let $M$ be an orientable manifold of real dimension $2n+1$.
An {\em almost CR structure} on $M$ is given by a corank one subbundle $Q$
of $TM$, the tangent bundle of $M$, together with a complex
structure $J : Q \rightarrow Q$.
Let $\onezero = \{Z \in Q \otimes {\Bbb C} \mid JZ = \rmo Z\}$; it is a complex rank $n$ subbundle of 
$TM \otimes{\Bbb C}$ satisfying $\onezero \cap \overline{\onezero} = \{0\}$. 
The bundle $\onezero$ recovers $Q$ and $J$ by $Q=\text{Re}(\onezero \oplus
\overline{\onezero})$ and $J(Z+\overline{Z}) = \rmo(Z-\overline{Z})$
for $Z \in \onezero$, respectively. 
A {\em CR structure} is an almost CR structure satisfying the integrability condition 
$[\Gamma(\onezero), \Gamma(\onezero)] \subset \Gamma(\onezero)$, 
or equivalently
\begin{equation}\label{CR_integrability0}
[X,Y]-[JX,JY],\,\, [X,JY]+[JX,Y] \in \Gamma(Q)
\end{equation}
and
\begin{equation}\label{CR_integrability}
J([X,Y]-[JX,JY])=[X,JY]+[JX,Y]
\end{equation}
for all $X, Y\in \Gamma(Q)$.
An almost CR structure is said to be {\em partially integrable} if it satisfies \eqref{CR_integrability0}, 
which is equivalent to the condition $[\Gamma(\onezero), \Gamma(\onezero)] \subset \Gamma(Q\otimes \C)$. 

Let $M$ be an almost CR manifold, and $\theta$ a one-form on $M$ whose kernel is 
the bundle of hyperplanes $Q$.
Such a $\theta$ exists globally, since we assume $M$ is orientable,
and $Q$ is oriented by its complex structure.
Associated with $\theta$ is a form $\Levi_\theta$ on $Q$, defined by
$$
\Levi_\theta(X,Y) = d\theta(X,JY),\quad X,Y \in Q, 
$$
and called the {\em Levi form} of $\theta$.
If the almost CR structure is partially integrable, then $\Levi_\theta$ 
is symmetric and $J$-invariant.
If $\theta$ is replaced by $\theta^\prime = \lambda\theta$ 
for a function $\lambda\neq 0$, then $\Levi_\theta$
changes conformally by $\Levi_{\theta^\prime} = \lambda \Levi_\theta$. 
An almost CR structure is said to be {\em \spc}~if it is partially integrable and 
$\Levi_\theta$ is positive or negative definite for some (hence any) choice of $\theta$.
In this case, $Q$ gives a contact structure on $M$, and $\theta$
is a contact form.

A CR structure typically arises on a real hypersurface $M$
of a complex manifold (of complex dimension $n+1$).
In this case $Q = TM \cap \mathcal{J}(TM)$ and $J = \mathcal{J}|_Q$, 
where $\mathcal{J}$ is the complex structure of the ambient complex manifold.
If $\rho$ is a defining function for $M$, then
$\theta = - \mathcal{J}(d\rho)/2$ annihilates $Q$.

A {\em pseudohermitian structure} on $M$ is a \spc almost CR structure together
with a choice of $\theta$ such that $\Levi_\theta$ is positive definite.
As $\theta$ is a contact form, it is accompanied by
the corresponding Reeb field $T$, determined by the equations
$\theta(T) = 1\quad \mbox{and}\quad d\theta(T,\cdot) = 0$. 

We now introduce a \qua analogue of CR structure.

\medskip\noindent
\Definition
Let $M$ be a connected, orientable manifold of dimension $4n+3$.
An {\em almost hyper CR structure} on $M$ is a pair of almost CR structures $(Q_1, I)$ 
and $(Q_2, J)$ which satisfies the following conditions:
\begin{enumerate}
\renewcommand{\theenumi}{\roman{enumi}}
\renewcommand{\labelenumi}{(\theenumi)}
\item 
$Q_1$ and $Q_2$ are transversal to each other;
\item the relation $IJ = -JI$ holds on $I(Q_1\cap Q_2) \cap J(Q_1\cap Q_2)$, the maximal domain 
on which the both sides make sense. 
\end{enumerate}

\medskip
We define the third almost CR structure $(Q_3, K)$ as follows. 
Set
$$
Q_3 = I(Q_1\cap Q_2) + J(Q_1\cap Q_2)\quad \mbox{and}\quad 
K= \left\{\begin{array}{cc}
-JI & \mbox{on $I(Q_1\cap Q_2)$},\\
IJ & \mbox{on $J(Q_1\cap Q_2)$}. 
\end{array}\right.
$$
Then $Q_3$ is a corank one subbundle of $TM$, and $K$ is well-defined
and satisfies the equation $K^2 = -\id$.
By the condition (ii), $Q_3$ is transversal to both $Q_1$ and $Q_2$. 
Moreover, the following relations hold: 
$$
I(Q_1 \cap Q_2) = Q_1 \cap Q_3,\quad J(Q_2 \cap Q_3) = Q_2 \cap Q_1,\quad K(Q_3 \cap Q_1) = Q_3 \cap Q_2;
$$
$$
\mbox{$IJ = K$ on $Q_2 \cap Q_3$,\quad $JI = -K$ on $Q_1 \cap Q_3$,\quad $JK = I$ on $Q_3 \cap Q_1$},
$$
$$
\mbox{$KJ = -I$ on $Q_2 \cap Q_1$,\quad $KI = J$ on $Q_1 \cap Q_2$,\quad $IK = -J$ on $Q_3 \cap Q_2$}.
$$

Set $Q = \cap_{a=1}^3 Q_a$.
It is a corank three subbundle of $TM$, and has three complex structures
$I$, $J$, $K$ satisfying the quaternion relations.
Henceforth, we shall write $I_1 = I$, $I_2 = J$ and $I_3 = K$ when
appropriate. 

\medskip\noindent
\Definition 
A triple $(T_1, T_2, T_3)$ of vector fields transverse to the subbundle $Q$
is called an {\em admissible triple} if it satisfies the following 
conditions:
$$ \hspace{-6cm}
\mbox{(i)}\,\, T_a \in \Gamma(Q_b \cap Q_c);\quad \mbox{(ii)}\,\, I_a T_b = T_c.
$$
We have
\begin{eqnarray*}
&& Q_a = Q \oplus {\Bbb R}T_b \oplus {\Bbb R}T_c, \\
&& TM = Q_a \oplus {\Bbb R}T_a = Q \oplus {\Bbb R}T_1 \oplus {\Bbb R}T_2 \oplus
{\Bbb R}T_3.
\end{eqnarray*} 
We call $Q^\perp = \oplus_{a=1}^3 \R T_a$ an {\em admissible three-plane
field}.

\medskip 
Note that an admissible triple $(T_1,T_2,T_3)$ certainly exists. 
Indeed, take $T_1\in \Gamma(Q_2 \cap Q_3)$ such that $(T_1)_q \notin Q_q$
($\Leftrightarrow (T_1)_q\notin (Q_1)_q$) for all $q\in M$.
Such a $T_1$ exists globally since $Q_2\cap Q_3$ is orientable and
$Q$ is oriented by its complex structures. 
Now it suffices to set $T_2 = KT_1$ and $T_3 = IT_2$. 

We shall next define an almost CR structure $(Q_{{\bf v}}, I_{{\bf v}})$
for each unit vector ${{\bf v}} = (v_1, v_2, v_3)\in \R^3$. 
Roughly speaking, $I_{{\bf v}}$ is defined to be $v_1I + v_2J +v_3K$,
which, however, makes sense only on $Q$. 
We rectify this defect by proceeding as follows.
Let $(T_1,T_2,T_3)$ be an admissible triple, and extend $I_1$, $I_2$, $I_3$ to endomorphisms 
$\widetilde{I_a} \colon TM \rightarrow TM$ by setting $\widetilde{I_a}T_a = 0$, 
and define $(Q_{{\bf v}}, I_{{\bf v}})$ by
\begin{equation}\label{Q_v}
Q_{{\bf v}} = Q\oplus \left\{x_1T_1 + x_2T_2 + x_3T_3 \mid x_1,x_2,x_3
\in {\Bbb R}, \sum_{a=1}^3 x_av_a = 0 \right\},
\end{equation}
\begin{equation}\label{I_v}
I_{{\bf v}} = (v_1\widetilde{I_1} + v_2\widetilde{I_2} + v_3\widetilde{I_3})|_{Q_{{\bf v}}}.
\end{equation} 
It is easy to verify that $I_{{\bf v}}$ indeed preserves $Q_{{\bf v}}$,
satisfies the equation ${I_{{\bf v}}}^2 = -\id$, and
$(Q_{{\bf v}}, I_{{\bf v}})$
is independent of the particular choice of admissible triple. 
Thus, associated with an almost hyper CR strucutre,
{{there}} is a canonical
family of almost CR structures parametrized by the unit sphere
$S^2$. 

Note that the above construction of the almost CR structure $(Q_{{\bf v}}, I_{{\bf v}})$ may be performed pointwise. 
Therefore, as ${\bf v}$, we can also take a variable function with values in $S^2$. 
 
There are two possible ways to define when a diffeomorphism between two almost hyper CR 
manifolds is an isomorphism. 
One way is to require that the diffeomorphism preserves each of the two almost CR 
structures 
constituting the almost hyper CR structure. 
The other is to require that it preserves the $S^2$-family of almost CR structures 
constructed above. 
We will find the effect of this difference when we investigate the automorphisms of 
the sphere $S^{4n+3}$ (\S 2, Example \ref{qsphere}). 


\medskip\noindent
\Definition 
An $\R^3$-valued one-form $\theta = (\theta_1, \theta_2, \theta_3)$ 
on an almost hyper CR manifold $M$ is said to be 
{\em compatible} with the almost hyper CR structure if it satisfies 
$$
\ker \theta_a = Q_a,\quad a = 1,2,3,
$$
\begin{equation}\label{eq101}
\theta_3\circ I = \theta_2\,\,\mbox{ on $Q_1$},\quad
\theta_1\circ J = \theta_3\,\,\mbox{ on $Q_2$},\quad
\theta_2\circ K = \theta_1\,\,\mbox{ on $Q_3$}.
\end{equation}

\medskip
Note that such a $\theta$ exists; it is enough to take an admissible triple
$(T_1, T_2, T_3)$ and
choose $\theta_a$ annihilating $Q_a$ so that $\theta_a(T_a)$ are nonzero and equal 
to each other (e.g., $\theta_a(T_a)= 1$). 
It is unique up to multiplication by a nowhere vanishing,
real-valued function.

In order to define a quaternionic analogue of Levi form, we require that our almost hyper CR structure 
should satisfy some sort of integrability condition. 

\medskip\noindent
\Definition An almost hyper CR structure is said to be {\em integrable} if it satisfies the following 
conditions for $a=1,2,3$ and for all $X,Y\in \Gamma(Q)$: 
\begin{equation}\label{integrability_condition0}
[X,Y] - [I_aX,I_aY] \in \Gamma(Q_a); 
\end{equation}
\begin{equation}\label{integrability_condition}
I_a ( [X,Y] - [I_aX,I_aY] ) - [X,I_aY] - [I_aX,Y]\in \Gamma(Q). 
\end{equation} 
Henceforth, we shall assume throughout that our almost hyper CR structure is integrable and refer to it 
as a {\em hyper CR structure}. 


\medskip\noindent
\Remark The integrability conditions \eqref{integrability_condition0}, \eqref{integrability_condition} are natural ones, 
as they are satisfied by the local hyper CR structure of any real hypersurface in a quaternionic manifold. 
See \S 2 for details. 

When a hyper CR structure is given, 
we can show that for any $S^2$-valued function ${\bf v}$, 
the almost CR structure $(Q_{{\bf v}}, I_{{\bf v}})$ defined above 
satisfies (appropriately modified versions of) \eqref{integrability_condition0}, 
\eqref{integrability_condition}. 
(see Proposiition \ref{familyintegrable} in Appendix \ref{ap_integrability}). 

\medskip 
Let $M$ be a hyper CR manifold and $\theta = (\theta_1, \theta_2, \theta_3)$ a compatible $\R^3$-valued one-form on $M$.
Note that by \eqref{integrability_condition0} we have 
\begin{equation}\label{I_a_invariance_of_dtheta_a}
d\theta_a(X,Y) = d\theta_a(I_aX,I_aY)
\end{equation} 
for $a=1,2,3$ and for all $X,Y\in \Gamma(Q)$. 
Moreover, we have the following identity for $X, Y\in Q$:
\begin{eqnarray}\label{eq102}
d\theta_1(X,IY) + d\theta_1(JX, KY) &=& d\theta_2(X,JY)
+ d\theta_2(KX, IY)\\
&=& d\theta_3(X, KY) + d\theta_3(IX, JY). \nonumber
\end{eqnarray}
Indeed, by plugging the both sides of \eqref{integrability_condition} with $a=1$ in $\theta_3$ and using \eqref{eq101}, 
we obtain
$$
\theta_2([X,Y]-[IX,IY]) = \theta_3([X,IY]+[IX,Y]),
$$
where $X, Y$ are extended to sections of $Q$. 
Therefore, 
$$
d\theta_2(X,Y)-d\theta_2(IX,IY) = d\theta_3(X,IY)+d\theta_3(IX,Y).
$$
Replacing $Y$ by $JY$ and using \eqref{I_a_invariance_of_dtheta_a}, we obtain the second equality of (\ref{eq102}).
We now define $\Levi_\theta(X,Y)$, the {\em Levi form} of $\theta$,
to be the half of this common quantity:
\begin{eqnarray*}
\Levi_\theta(X,Y) &=& \frac{1}{2}(d\theta_1(X,IY) + d\theta_1(JX, KY))\\
&=& \frac{1}{2} (d\theta_2(X,JY) + d\theta_2(KX, IY) )\\
&=& \frac{1}{2} (d\theta_3(X, KY) + d\theta_3(IX, JY) ).
\end{eqnarray*}
Note that $\Levi_\theta$ is nothing but 
the quaternion-hermtian (that is, symmetric and invariant under $I$, $J$, $K$) part 
of the ``complex" Levi form $\Levi_{\theta_a} = d\theta_a(\cdot, I_a \cdot)$ restricted to $Q$.
If $\theta$ is replaced by $\theta^\prime = \lambda \theta$, $\lambda \neq 0$,
then $\Levi_\theta$ changes conformally by $\Levi_{\theta^\prime}
= \lambda \Levi_\theta$.

\medskip\noindent
\Definition
We say that a hyper CR structure is {\em \spc}~if $\Levi_\theta$ is positive or negative definite
for some (hence any) choice of $\theta$. 
A {\em hyper pseudohermitian structure} is a \spc hyper CR structure together with
a choice of $\theta$ such that $\Levi_\theta$ is positive definite. 
We also call, by abuse, such a $\theta$ a {\em pseudohermitian structure}. 

\medskip 
We now introduce another \qua analogue of CR structure.

\medskip\noindent
\Definition
A {\em \qua CR structure} on $M$ is a covering of $M$ by local hyper CR
structures which satisfies the following condition: 
let
$\{(Q_a, I_a)\}_{a=1,2,3}$ and $\{(Q_a^\prime, I_a^\prime)\}_{a=1,2,3}$ 
be two such local structures defined on open subsets
$U$ and $U^\prime$ respectively.
If $U\cap U^\prime\neq\emptyset$, there is an 
$SO(3)$-valued function $S=S_{UU^\prime}\colon U\cap U^\prime \rightarrow SO(3)$ 
such that 
\begin{equation}\label{qcr}
Q_{\bf v}^\prime = Q_{S{\bf v}},\quad  
I_{\bf v}^\prime = I_{S{\bf v}}, \quad {\bf v}\in S^2, 
\end{equation}
where the notation is as in \eqref{Q_v}, \eqref{I_v}. 
(Note that $S{\bf v}$ is a variable function of $q\in U\cap U'$.) 

\medskip
There is a double covering $Sp(1)\rightarrow SO(3)$, and if $S$ can be lifted to 
an $Sp(1)$-valued function $\sigma\colon U\cap U^\prime \rightarrow Sp(1)$, which is the case 
when $U\cap U^\prime$ is simply-connected, then \eqref{qcr} may be written as 
\begin{equation}\label{qcr2}
Q_{\bf v}^\prime = Q_{\sigma^{-1}{\bf v}\sigma},
\quad
I_{\bf v}^\prime = I_{\sigma^{-1}{\bf v}\sigma},
\quad {\bf v}\in S^2.
\end{equation}
Here, $\sigma^{-1}{\bf v}\sigma$ is computed by regarding ${\bf v}$ as an imaginary quaternion 
via the identification $\R^3 = \im \H$. 
We adopt the convention that the atlas defining a \qua
CR structure is extended to a maximal one.
In particular, any (global) hyper CR structure canonically determines a \qua CR structure. 
Henceforth, we shall regard a hyper CR manifold as
equipped with this \qua CR structure.

Given a \qua CR structure, there are local corank three bundles $Q_U$
associated with the local hyper CR structures.
But $Q_U = Q_{U^\prime}$ on $U\cap U^\prime$, and they give rise to
a bundle $Q$ defined globally on $M$.

Let $\Theta = \{\theta_U \}$ be a collection of local $\R^3$-valued
one-forms compatible with the local hyper CR structures such that
\begin{equation}\label{theta_deform}
(\theta_{U^\prime})_a = \sum_{b=1}^3 s_{ab} (\theta_U)_b,\quad a=1,2,3
\end{equation}
on $U\cap U^\prime$, where $S=(s_{ab})$ is the $SO(3)$-valued function as in the definition above.
Such a collection exists, and it is unique up to multiplication
by a nowhere vanishing, real-valued function.
Associated with $\theta_U$ are the local Levi forms $\Levi_{\theta_U}$, 
for which we have the following 

\begin{Proposition}\label{prepLevi}
Let $M$ be a quaternionic CR manifold, and $\Theta = \{\theta_U\}$ a collection 
of local $\R^3$-valued one-forms on $M$ as above. 
Then the local Levi forms $\Levi_{\theta_U}$ and $\Levi_{\theta_{U'}}$ coincide on $U \cap U'$.
\end{Proposition}

The proof of this proposition will be given in Appendix \ref{ap_Levi}. 
By Proposition \ref{prepLevi}, we obtain a globally defined symmetric bilinear form,
denoted by $\Levi_\Theta$, and call it the {\em Levi form} of $\Theta$.
Using this we define the
{\em strong pseudoconvexity} of a \qua CR structure as before. 
A {\em quaternionic pseudohermitian structure} is a \spc quaternionic CR structure together with
a choice of a collection $\Theta$ such that $\Levi_\Theta$ is positive definite.
Again, by abuse, such a collection $\Theta$ is called a {\em pseudohermitian structure}. 

Each fibre of $Q$ has a family of complex structures parametrized
by the two-sphere $S^2$ with no preferred choice of triple satisfying
the quaternion relations.
This amounts to saying that the bundle $Q$ has a $GL(n,{\Bbb H})
\cdot {\Bbb H}^*$-structure, where
$GL(n,{\Bbb H}) \cdot {\Bbb H}^* = GL(n,{\Bbb H}) \times Sp(1)
/\{\pm I_{n+1}\}$.
A choice of a pseudohermitian structure $\Theta = \{\theta_U\}$ gives $Q$ a fibre metric
$\Levi_\Theta$, and it is invariant under any of the complex
structures on the fibre.
Thus the choice of $\Theta$ reduces the structure group
of $Q$ from $GL(n,{\Bbb H}) \cdot {\Bbb H}^*$ to
$Sp(n)\cdot Sp(1) = Sp(n)\times Sp(1)/\{\pm I_{n+1}\}$.

\section{Real hypersurface and examples}
\subsection{Real hypersurface} 

Let $\mathcal{N}$ be a quaternionic manifold of dimension $4n+4$; thus $\mathcal{N}$ admits a torsion-free 
$GL(n+1, \H)\cdot \H^*$-affine connection $\mathcal{D}$. 
Then, in a neighborhood $\mathcal{U}$ of any point of $\mathcal{N}$, there exist almost complex structures 
$\mathcal{I}_a$, $a=1,2,3$, which satisfy 
$\mathcal{D} \mathcal{I}_a = \sum_{b=1}^3 \gamma_{ab}\otimes \mathcal{I}_b$, 
where $\gamma_{ab}$ are one-forms on $\mathcal{U}$ satisfying $\gamma_{ab} = -\gamma_{ba}$. 
Let $M$ be a real hypersurface in $\mathcal{N}$. 
Then $M$ comes equipped with a quaternionic CR structure in a canonical manner, 
by setting $U = M\cap \mathcal{U}$ for each $\mathcal{U}$ as above and defining 
$Q_a = TU\cap \mathcal{I}_a(TU)$ and $I_a = \mathcal{I}_a|_{Q_a}$. 
Thus the corank three subbundle $Q$ is given by $Q|_U = TU \cap \mathcal{I}_1(TU) \cap \mathcal{I}_2(TU) 
\cap \mathcal{I}_3(TU)$.
It remains to verify that the integrability conditions \eqref{integrability_condition0}, 
\eqref{integrability_condition} hold for all $X,Y\in \Gamma(Q|_U)$. 
To see that \eqref{integrability_condition0} holds, it suffices to show 
\begin{equation}\label{realhyp_formula1}
\mathcal{I}_a([X, Y] - [I_aX, I_aY])\in \Gamma(TU).
\end{equation} 
For this, let $\mathcal{X}, \mathcal{Y}\in \Gamma(T\mathcal{U})$ be local extensions of $X,Y$ respectively, 
and compute 
\begin{eqnarray*} 
\lefteqn{\mathcal{I}_a([\mathcal{X}, \mathcal{Y}] - [\mathcal{I}_a\mathcal{X}, \mathcal{I}_a\mathcal{Y}])}\\ 
&=& \mathcal{I}_a ( \mathcal{D}_{\mathcal{X}} \mathcal{Y} - \mathcal{D}_{\mathcal{Y}} \mathcal{X} 
- \mathcal{D}_{\mathcal{I}_a\mathcal{X}} \mathcal{I}_a\mathcal{Y} + \mathcal{D}_{\mathcal{I}_a\mathcal{Y}} 
\mathcal{I}_a\mathcal{X} ) \\ 
&=& \mathcal{D}_{\mathcal{X}} \mathcal{I}_a\mathcal{Y} - \sum_b \gamma_{ab}(\mathcal{X}) \mathcal{I}_b \mathcal{Y}  
- \mathcal{D}_{\mathcal{Y}} \mathcal{I}_a\mathcal{X} + \sum_b \gamma_{ab}(\mathcal{Y}) \mathcal{I}_b \mathcal{X} \\ 
&& + \mathcal{D}_{\mathcal{I}_a\mathcal{X}} \mathcal{Y}
- \sum_b \gamma_{ab}(\mathcal{I}_a\mathcal{X}) \mathcal{I}_a\mathcal{I}_b\mathcal{Y} 
- \mathcal{D}_{\mathcal{I}_a\mathcal{Y}} \mathcal{X}
+ \sum_b \gamma_{ab}(\mathcal{I}_a\mathcal{Y}) \mathcal{I}_a\mathcal{I}_b\mathcal{X} \\ 
&=& [\mathcal{X}, \mathcal{I}_a \mathcal{Y}] + [\mathcal{I}_a\mathcal{X}, \mathcal{Y}]
- \sum_b \gamma_{ab}(\mathcal{X}) \mathcal{I}_b \mathcal{Y}  
+ \sum_b \gamma_{ab}(\mathcal{Y}) \mathcal{I}_b \mathcal{X} \\ 
&& 
- \sum_b \gamma_{ab}(\mathcal{I}_a\mathcal{X}) \mathcal{I}_a\mathcal{I}_b\mathcal{Y} 
+ \sum_b \gamma_{ab}(\mathcal{I}_a\mathcal{Y}) \mathcal{I}_a\mathcal{I}_b\mathcal{X}. 
\end{eqnarray*}
Restricted to $U$, this shows 
$$
\mathcal{I}_a( [X, Y] - [I_aX, I_aY]) = [X, I_aY]+[I_aX, Y]\quad \mod \Gamma(Q|_U).  
$$
Therefore, \eqref{realhyp_formula1} holds, and the second condition \eqref{integrability_condition} 
for integrability is also verified. 

We now restrict ourselves to the case that the quaternionic manifold $\mathcal{N}$ is a queternionic 
affine space ${\Bbb H}^{n+1}$. 
Let $M$ be a (local) real hypersurface in ${\Bbb H}^{n+1}$,
and $\rho$ a defining function for $M$:
$M = \rho^{-1}(0)$, $d\rho \neq 0$ along $M$. 
Note that the tangent spaces of $M$ are given by
$T_q M = \{X \in {\Bbb H}^{n+1} \mid
d\rho_q(X)=0 \}$, $q\in M$.
Here and throughout, the tangent spaces $T_q {\Bbb H}^{n+1}$
are identified with ${\Bbb H}^{n+1}$ in the standard manner.
For each unit imaginary quaternion ${\bf v}$, a parallel complex structure 
$\mathcal{I}_{\bf v}$ on ${\Bbb H}^{n+1}$ is defined by
$\mathcal{I}_{\bf v} X = X {\bf v}^{-1}$, $X\in {\Bbb H}^{n+1}$. 
Thus there is a family of complex structures on ${\Bbb H}^{n+1}$
parametrized by $S^2$, the unit sphere in ${\rm Im}{\Bbb H} = \R^3$.
Each of these complex structures, $\mathcal{I}_{\bf v}$, determines
an (integrable) CR structure
$(Q_{\bf v}, I_{\bf v})$ on $M$, 
where $Q_{\bf v} = TM \cap \mathcal{I}_{\bf v}(TM)$ 
and $I_{\bf v} = \mathcal{I}_{\bf v}|_{Q_{\bf v}}$. 
In particular, the CR structures $(Q_1, I) = (Q_{\bf i}, I_{\bf i})$,
$(Q_2, J) = (Q_{\bf j}, I_{\bf j})$ and $(Q_3, K) = (Q_{\bf k},
I_{\bf k})$ define a hyper CR structure on $M$. 
Note that the condition that the three CR structures are all integrable is stronger 
than the integrability conditons \eqref{integrability_condition0}, \eqref{integrability_condition}. 

For the computation below, it is convenient to introduce the complex
coordinates
$z_h= x_h^0+\rmo x_h^1$, $w_h = x_h^2+\rmo x_h^3$, $h=1,\dots,n+1$, 
where $q = (q_1,\dots,q_{n+1})$ and
$q_h = x_h^0 - x_h^1{\bf i} - x_h^2{\bf j} - x_h^3{\bf k}$.
We have
\begin{equation}\label{H^{n+1}_hypercomplex_structure}
\begin{array}{c}
\mathcal{I}(dz_h) = \rmo dz_h,\,\, \mathcal{I}(dw_h) = \rmo dw_h;\quad 
\mathcal{J}(dz_h) = -d\overline{w_h},\,\, \mathcal{J}(dw_h) = d\overline{z_h};\\ 
\mathcal{K}(dz_h) = -\rmo d\overline{w_h},\,\, \mathcal{K}(dw_h) = \rmo d\overline{z_h}.
\end{array}
\end{equation}

As in the complex CR case, let $\theta_a = - \mathcal{I}_a (d\rho)/2$, $a=1,2,3$. 
Then $\theta = (\theta_1, \theta_2, \theta_3)$ is compatible with the hyper CR structure of $M$.
In the complex coordinates,
$$
\theta_1 = \frac{\rmo}{2} \sum_{h=1}^{n+1}\left(
- \frac{\partial \rho}{\partial z_h} dz_h
+ \frac{\partial \rho}{\partial \overline{z_h}} d\overline{z_h}
- \frac{\partial \rho}{\partial w_h} dw_h
+ \frac{\partial \rho}{\partial \overline{w_h}} d\overline{w_h}
\right),
$$
$$
\theta_2 + \rmo \theta_3
= \sum_{h=1}^{n+1} \left(
- \frac{\partial \rho}{\partial \overline{w_h}} dz_h
+ \frac{\partial \rho}{\partial \overline{z_h}} dw_h
\right).
$$
One can verify that the Levi form of $\theta$
is given by
\begin{eqnarray}\label{eq301}
\Levi_\theta &=& \sum_{h,l=1}^{n+1}
\left[\left(
\frac{\partial^2 \rho}{\partial z_h \partial \overline{z_l}}
+ \frac{\partial^2 \rho}{\partial \overline{w_h} \partial w_l}
\right) \left(dz_h\cdot d\overline{z_l} + d\overline{w_h}\cdot
dw_l\right)\right. \nonumber \\
&&\left.
+ \frac{\partial^2 \rho}{\partial z_h \partial\overline{w_l}}
\left(dz_h\cdot d\overline{w_l} - d\overline{w_h}\cdot
dz_l\right)
+ \frac{\partial^2 \rho}{\partial\overline{z_h} \partial w_l}
\left(d\overline{z_h}\cdot dw_l - dw_h\cdot d\overline{z_l}\right)
\right].
\end{eqnarray}

\bigskip\noindent
\Example\label{qsphere} Let $\qsphere = \left\{ q\in {\Bbb H}^{n+1} \mid |q|^2
= \overline{q}\cdot q = 1 \right\}$ be the unit sphere in ${\Bbb H}^{n+1}$, 
where $v\cdot w = \sum_{h=1}^{n+1}v_h w_h$ for $v = (v_1,\dots,v_{n+1})$,
$w = (w_1,\dots,w_{n+1}) \in {\Bbb H}^{n+1}$.
As a real hypersurface in ${\Bbb H}^{n+1}$, $\qsphere$ is endowed with
a hyper CR structure, whose underlying corank three bundle $Q$ is given by
$Q_q = \left\{ X\in {\Bbb H}^{n+1} \mid \overline{q}\cdot X = 0
\right\}$, $q\in \qsphere$. 
A standard choice of defining function for $\qsphere$ is $\rho(q)=|q|^2-1$,
and then the corresponding $\R^3$-valued one-form is given by
$\theta_S = \sum_{h=1}^{n+1} ( d\overline{q_h} q_h - \overline{q_h}dq_h)/2$, 
where we identify $\R^3={\rm Im}{\Bbb H}$. 
By (\ref{eq301}), we see that the Levi form
$\Levi_{\theta_S}$ is twice the standard Riemannian metric of $\qsphere$
restricted to $Q$.
In particular, the hyper CR structure of $\qsphere$ is strongly pseudoconvex.
We shall refer to $\theta_S$ as the {\em standard pseudohermitian structure} of $\qsphere$.

The sphere $\qsphere$ may be regarded as the boundary at infinity of 
quaternionic hyperbolic space $\qhyp$. 
The isometry group $G$ of $\qhyp$ is given by $G = Sp(n+1,1)/\{\pm I_{n+2}\}$.
As $G$ acts on $\qhyp$ transitively, 
we have a representation of $\qhyp$ as a coset space $G/K$, where
$K = Sp(n+1)\cdot Sp(1) = Sp(n+1)\times Sp(1)/\{\pm I_{n+2}\}$,
a maximal compact subgroup of $G$.
The action of $G$ extends to $\qsphere$, and we shall examine how it
transforms the hyper CR structure of $\qsphere$.
We represent $\gamma \in G$ by a matrix
$$
\left(\begin{array}{cc}a&b\\c&d\end{array}\right) \in Sp(n+1, 1),
$$
where $a^* a - c^* c = I_{n+1}$, $b^*b-d^*d = -1$, $a^*b-c^*d = 0$. 
The action of $\gamma$ on $\qsphere$ is then given by
$q = \mbox{}^t(q_1,\dots,q_{n+1}) \mapsto (aq+b)(cq+d)^{-1}$. 
By direct calculation, we obtain
\begin{equation}\label{groupaction}
\gamma^*\theta_S = \frac{1}{|cq+d|^2}\left(\frac{cq+d}{|cq+d|}\right)
\theta_S \left(\frac{cq+d}{|cq+d|}\right)^{-1}.
\end{equation}
This formula means that for each unit imaginary quaternion ${\bf v} \in S^2$, 
the CR structure $(Q_{\bf v}, I_{\bf v})$ is transformed as
$\gamma^* Q_{\bf v} = Q_{\sigma_\gamma^{-1}{\bf v}\sigma_\gamma}$, 
$\gamma^* I_{\bf v} = I_{\sigma_\gamma^{-1}{\bf v}\sigma_\gamma}$, 
where $\sigma_\gamma$ is the function of $q$ given by $\sigma_\gamma(q) = (cq+d)/|cq+d|$. 
Thus $G$ preserves the \qua CR structure of $\qsphere$. 
Likewise, if $\gamma\in K$, then $c=0$ and $\sigma_\gamma$ is constant. 
Therefore, $K$ preserves the canonical $S^2$-family of CR structures associated with the 
hyper CR structure of $\qsphere$. 
If $\gamma\in Sp(n)$ further, then $d= \pm 1$ and $\sigma_\gamma = \pm 1$. 
Therefore, $Sp(n)$ preserves each of the two CR structures constituting the hyper CR 
structure. 

For ${\bf v} \in S^2$, let $T_{\bf v}$ be the
vector field on $\qsphere$ defined by
$(T_{\bf v})_q = q{\bf v}^{-1}$, $q \in \qsphere$, 
and let $T_1 = T_{\bf i}$, $T_2 = T_{\bf j}$ and $T_3 = T_{\bf k}$.
One can check that $T_a$ satisfies $\theta_a(T_a) = 1$ and
$d\theta_a(T_a,X) = 0$ for all $X\in Q$.
The action of $\gamma\in G$ transforms $T_{\bf v}$ as
\begin{equation}\label{transform2}
\gamma^*T_{\bf v} = \left(\gamma^{-1}\right)_*T_{\bf v} 
= e^{-2f} \left[ T_{\sigma_\gamma^{-1}{\bf v}\sigma_\gamma} 
- 2 I_{\sigma_\gamma^{-1}{\bf v}\sigma_\gamma} d_bf^\# \right],
\end{equation}
where $f = -\log |cq+d|$, 
$d_bf$ denotes the restriction of $df$ to $Q$ and $d_bf^\#$ 
is the section of $Q$ dual to $d_bf$ with respect to $\Levi_{\theta_S}$ 
(cf.~\S 3, Remark \ref{canonical_triple_S^{4n+3}}). 

\bigskip\noindent
\Example Let $E$ be the real ellipsoid 
$$
E\,\, : \,\, \sum_{i=1}^{n+1} \left( a_i \left({z_i}^2 + {\overline{z_i}}^2 \right)
+ b_i z_i \overline{z_i}
+ c_i \left( {w_i}^2 + {\overline{w_i}}^2 \right) + d_i w_i \overline{w_i} 
\right) - 1 = 0, 
$$
where $a_i,\cdots, d_i$ are real, $b_i, d_i>0$ and we are using the complex coordinates 
$z_i$, $w_i$ of $\H^{n+1}$ to write the defining equation. 
Let $\theta = (\theta_1, \theta_2, \theta_3)$ be the $\R^3$-valued one-form 
compatible with the hyper CR structure of $E$, 
corresponding to the defining function chosen to be the left-hand side of 
the defining equation.  
By (\ref{eq301}), the Levi form is 
$$
\Levi_\theta = \sum_{i=1}^{n+1}
\left( b_i + d_i \right) \left(dz_i\cdot d\overline{z_i} + d\overline{w_i}\cdot
dw_i\right) 
$$
restricted to $Q$, and therefore the hyper CR structure of $E$ is strongly pseudoconvex. 

On the other hand, the complex Levi forms are 
\begin{eqnarray*}
\Levi_{\theta_1} &=& 2 \sum_{i=1}^{n+1} (b_i\, dz_i \cdot d\overline{z_i}
+ d_i\, d\overline{w_i} \cdot dw_i),\\
\Levi_{\theta_2} &=& \re \sum_{i=1}^{n+1} [
(b_i + d_i) ( dz_i\cdot d\overline{z_i} + d\overline{w_i} \cdot dw_i )\\
&& + 2 (a_i + c_i) ( dz_i\cdot dz_i + d\overline{w_i} \cdot d\overline{w_i} )],\\
\Levi_{\theta_3} &=& \re \sum_{i=1}^{n+1} [
(b_i + d_i) ( dz_i\cdot d\overline{z_i} + d\overline{w_i} \cdot dw_i )\\
&& + 2 (a_i - c_i) ( dz_i\cdot dz_i - d\overline{w_i} \cdot d\overline{w_i} )], 
\end{eqnarray*}
and these almost never coincide on $Q$; indeed, they coincide on $Q$ 
if and only if $b_i = d_i$ and $a_i = c_i = 0$ hold for $i=1,\dots,n+1$, 
that is, $E$ should be a ``quaternionic ellipsoid". 
We will revisit this example in \S 5, where we compare the quaternionic CR structure 
to Biquard's quaternionic contact structure \cite{biq}. 

\subsection{Quaternionic Heisenberg group and its deformation}

\Example\label{qheisen}
The {\em quaternionic Heisenberg group} $\qhgroup$ is the Lie group whose 
underlying manifold is ${\Bbb H}^n \times \im{\Bbb H}$ with coordinates
$(p,\tau) = (p_1,\dots,p_n, \tau)$ and whose group law is given by
$$
(p,\tau)\cdot(p^\prime, \tau^\prime) = \left(p+p^\prime,
\tau + \tau^\prime + \left(\overline{p}\cdot p^\prime
- \overline{p^\prime}\cdot p\right) \right).
$$
(There is a concise treatment of the complex Heisenberg group as
a CR manifold in \cite{jele}.)
\par
Write $p_\alpha = x_\alpha^0 - x_\alpha^1 {\bf i} - x_\alpha^2 {\bf j}
- x_\alpha^3 {\bf k}$
and $\tau = t_1 {\bf i} + t_2 {\bf j} + t_3 {\bf k}$.
The vector fields
{\Small
\begin{eqnarray*}
X_\alpha^0 &=&  \frac{\partial}{\partial x_\alpha^0}
             + 2x_\alpha^1 \frac{\partial}{\partial t_1}
             + 2x_\alpha^2 \frac{\partial}{\partial t_2}
             + 2x_\alpha^3 \frac{\partial}{\partial t_3},\quad
X_\alpha^1 \,\,=\,\,  \frac{\partial}{\partial x_\alpha^1}
             - 2x_\alpha^0 \frac{\partial}{\partial t_1}
             - 2x_\alpha^3 \frac{\partial}{\partial t_2}
             + 2x_\alpha^2 \frac{\partial}{\partial t_3},\\
X_\alpha^2 &=&  \frac{\partial}{\partial x_\alpha^2}
             + 2x_\alpha^3 \frac{\partial}{\partial t_1}
             - 2x_\alpha^0 \frac{\partial}{\partial t_2}
             - 2x_\alpha^1 \frac{\partial}{\partial t_3},\quad
X_\alpha^3 \,\,=\,\,  \frac{\partial}{\partial x_\alpha^3}
             - 2x_\alpha^2 \frac{\partial}{\partial t_1}
             + 2x_\alpha^1 \frac{\partial}{\partial t_2}
             - 2x_\alpha^0 \frac{\partial}{\partial t_3},\\
T_1 &=& 2 \frac{\partial}{\partial t_1},\quad
T_2 \,\,=\,\, 2 \frac{\partial}{\partial t_2},\quad
T_3 \,\,=\,\, 2 \frac{\partial}{\partial t_3}
\end{eqnarray*}
}
are left-invariant.
Let
$$
Q = \text{span} \{X_\alpha^a\}_{1\leq \alpha\leq n, 0\leq a\leq 3},\quad 
Q_a = Q \oplus {\Bbb R}T_b\oplus {\Bbb R}T_c, 
$$ 
and define complex structures $I_a$ on $Q_a$ by
$$
I_aX_\alpha^0 = X_\alpha^a,\,\, I_aX_\alpha^b = X_\alpha^c,\,\, 
I_aT_b = T_c. 
$$
Then the triple of (integrable) CR structures $(Q_a,I_a)$ gives a left-invariant
hyper CR structure on $\qhgroup$.
The $\im {\Bbb H}$-valued one-form
$$
\theta_H = \frac{1}{2}\left[d\tau + \sum_{\alpha=1}^n
\left(d\overline{p_\alpha} p_\alpha - \overline{p_\alpha}
dp_\alpha \right)\right]
$$
is left-invariant and compatible with the hyper CR structure.
The Levi form of $\theta_H$ is given by $\Levi_{\theta_H}(X_\alpha^a,
X_\beta^b) = 2\delta_{\alpha\beta}\delta_{ab}$.
Hence the hyper CR structure of $\qhgroup$ is strongly pseudoconvex. 
We shall refer to $\theta_H$ as the {\em standard pseudohermitian structure} of $\qhgroup$.
It is also worthwhile to mention that in this example the CR structures $(Q_a, I_a)$ are 
{\em not} strongly pseudoconvex. 

\medskip\noindent
\Remark The quaternionic CR and pseudohermitian structures of $\qsphere$ and $\qhgroup$ 
are related as follows. 
Let ${\frak P} = \left\{ (p^\prime, p_{n+1}) \in {\Bbb H}^{n+1}
\mid \re\,p_{n+1} = |p^\prime|^2 \right\}$. 
The mapping
$$
(q_1,\dots,q_{n+1}) \in \qsphere\setminus\{(0,\dots,0,-1)\} 
\mapsto
\left(q_1,\dots,q_n, 1-q_{n+1}\right) (1+q_{n+1})^{-1}\in
{\frak P}
$$
is a quaternionic analogue of Cayley transform. 
This mapping composed with
$(p^\prime,p_{n+1}) \mapsto (p^\prime, p_{n+1} - |p^\prime|^2)$
gives the equivalence mapping
\begin{eqnarray*}
&&F : (q_1,\dots,q_{n+1}) \in \qsphere\setminus\{(0,\dots,0,-1)\}\\
&&\mbox{}\mapsto
\left(q_1(1+q_{n+1})^{-1},\dots,q_n(1+q_{n+1})^{-1},(1+q_{n+1})^{-1}
- \left(1+\overline{q_{n+1}}\right)^{-1} \right) \in \qhgroup
\end{eqnarray*}
between \qua CR manifolds.
Hence $F^*\theta_H$ is a pseudohermitian structure (singular at
$(0,\dots,0,-1)$) for the \qua CR structure of $\qsphere$,
and thus has the form $\lambda \sigma \theta_S \sigma^{-1}$,
where $\lambda$ and $\sigma$ are respectively positive
and $Sp(1)$-valued functions.
Explicitly, we have
$\lambda = 1/|1+q_{n+1}|^2$, $\sigma = (1+q_{n+1})/|1+q_{n+1}|$. 

\Example\label{deformation_qheisen}
The hyper pseudohermitian structure of quaternionic Heisenberg group 
can be deformed by changing the definition of vector fields $X_\alpha^a$
as follows:
\begin{eqnarray*}
X_\alpha^0 &=&  \frac{\partial}{\partial x_\alpha^0}
             + A^1_\alpha x_\alpha^1 \frac{\partial}{\partial t_1}
             + A^2_\alpha x_\alpha^2 \frac{\partial}{\partial t_2}
             + A^3_\alpha x_\alpha^3 \frac{\partial}{\partial t_3},\\
X_\alpha^1 &=&  \frac{\partial}{\partial x_\alpha^1}
             - B^1_\alpha x_\alpha^0 \frac{\partial}{\partial t_1}
             - B^2_\alpha x_\alpha^3 \frac{\partial}{\partial t_2}
             + B^3_\alpha x_\alpha^2 \frac{\partial}{\partial t_3},\\
X_\alpha^2 &=&  \frac{\partial}{\partial x_\alpha^2}
             + C^1_\alpha x_\alpha^3 \frac{\partial}{\partial t_1}
             - C^2_\alpha x_\alpha^0 \frac{\partial}{\partial t_2}
             - C^3_\alpha x_\alpha^1 \frac{\partial}{\partial t_3},\\
X_\alpha^3 &=&  \frac{\partial}{\partial x_\alpha^3}
             - D^1_\alpha x_\alpha^2 \frac{\partial}{\partial t_1}
             + D^2_\alpha x_\alpha^1 \frac{\partial}{\partial t_2}
             - D^3_\alpha x_\alpha^0 \frac{\partial}{\partial t_3}, 
\end{eqnarray*}
where $A^a_\alpha, B^a_\alpha, C^a_\alpha, D^a_\alpha$ are real constants. 
We define vector fields $T_a$, bundles $Q$, $Q_a$ and complex structures $I_a$ on $Q_a$ 
as in Example \ref{qheisen}.
Then the triple of almost CR structures $(Q_a, I_a)$ gives 
an almost hyper CR structure on $\R^{4n}\times \R^3$, and  
always satisfies \eqref{integrability_condition0}. 
It satisfies 
\eqref{integrability_condition} (or more strongly, $(Q_a, I_a)$ are 
CR structures) if and only if 
$A^a_\alpha + B^a_\alpha + C^a_\alpha + D^a_\alpha$ 
does not depend on $a$ (though may depend on $\alpha$). 
Henceforth we assume this condtion is satisfied, and therefore, 
we obtain a hyper CR structure. 
Then the $\R^3$-valued one-form $\theta = (\theta_1,\theta_2,\theta_3)$ 
given by 
\begin{eqnarray*}
\theta_1 &=& \frac{1}{2} \left[ dt_1 + \sum_\alpha 
\left( -A^1_\alpha x^1_\alpha dx^0_\alpha
+ B^1_\alpha x^0_\alpha dx^1_\alpha - C^1_\alpha x^3_\alpha dx^2_\alpha 
+D^1_\alpha x^2_\alpha dx^3_\alpha \right) \right],\\ 
\theta_2 &=& \frac{1}{2} \left[ dt_2 + \sum_\alpha 
\left( -A^2_\alpha x^2_\alpha dx^0_\alpha + B^2_\alpha x^3_\alpha dx^1_\alpha
+ C^2_\alpha x^0_\alpha dx^2_\alpha - D^2_\alpha x^1_\alpha dx^3_\alpha \right) \right],\\
\theta_3 &=& \frac{1}{2} \left[ dt_2 + \sum_\alpha 
\left( -A^3_\alpha x^3_\alpha dx^0_\alpha - B^3_\alpha x^2_\alpha dx^1_\alpha
+ C^3_\alpha x^1_\alpha dx^2_\alpha + D^3_\alpha x^0_\alpha dx^3_\alpha \right) \right]
\end{eqnarray*} 
is compatible with the hyper CR structure and satisfies 
$\theta_a(T_b) = \delta_{ab}$. 
The complex Levi forms are 
\begin{eqnarray*}
\Levi_{\theta_1} &=& \frac{1}{2} \sum_\alpha \left\{
\left(A^1_\alpha+B^1_\alpha\right)\left(\left(dx^0_\alpha\right)^2 
+ \left(dx^1_\alpha\right)^2\right)
+ \left(C^1_\alpha+D^1_\alpha\right)\left(\left(dx^2_\alpha\right)^2 
+ \left(dx^3_\alpha\right)^2\right) \right\},\\ 
\Levi_{\theta_2} &=& \frac{1}{2} \sum_\alpha \left\{
\left(A^2_\alpha+C^2_\alpha\right)\left(\left(dx^0_\alpha\right)^2 
+ \left(dx^2_\alpha\right)^2\right)
+ \left(B^2_\alpha+D^2_\alpha\right)\left(\left(dx^3_\alpha\right)^2 
+ \left(dx^1_\alpha\right)^2\right) \right\},\\ 
\Levi_{\theta_3} &=& \frac{1}{2} \sum_\alpha \left\{
\left(A^3_\alpha+D^3_\alpha\right)\left(\left(dx^0_\alpha\right)^2 
+ \left(dx^3_\alpha\right)^2\right)
+ \left(B^3_\alpha+C^3_\alpha\right)\left(\left(dx^1_\alpha\right)^2 
+ \left(dx^2_\alpha\right)^2\right) \right\}, 
\end{eqnarray*} 
and the Levi form is
$$
\Levi_{\theta} = \frac{1}{4} \sum_\alpha\Lambda_\alpha 
\left( \left(dx^0_\alpha\right)^2 + \left(dx^1_\alpha\right)^2 
+ \left(dx^2_\alpha\right)^2 + \left(dx^3_\alpha\right)^2 \right), 
$$
where we set $\Lambda_\alpha = A^a_\alpha + B^a_\alpha + C^a_\alpha + D^a_\alpha$. 
Suppose now that $\Lambda_\alpha>0$ for all $\alpha$, so that 
the hyper CR structure is strongly pseudoconvex. 
We will revisit this example in \S 3, \S 5. 

\subsection{Principal bundle over a hypercomplex manifold}
Let $(N,I,J,K)$ be a hypercomplex manifold of real dimension $4n$,
that is, $I,J,K$ are (integrable) complex structures on $N$ satisfying $IJ=-JI=K$. 
Write $I_1 = I$, $I_2 = J$ and $I_3 = K$.
Let $G$ be a Lie group of dimension three, and 
let $\pi:M \rightarrow N$ be a principal $G$-bundle over $N$ 
with connection form $\theta$ and the corresponding curvature form $\Omega$.
Via an idetification of the Lie algebra ${\mathfrak g}$ of $G$ with ${\mathbb R}^3$ 
as vector spaces, we write $\theta$ and $\Omega$
as $\theta = (\theta_1,\theta_2,\theta_3)$ and 
$\Omega=(\pi^*\Omega_1, \pi^*\Omega_2, \pi^*\Omega_3)$, respectively.
Assume the following two conditions: 
\begin{enumerate}
\renewcommand{\theenumi}{\roman{enumi}}
\renewcommand{\labelenumi}{(\theenumi)}
\item The two-forms $\Omega_a$ on $N$ are invariant by $I_a$: 
$$
\Omega_a(I_aX,I_aY) = \Omega_a(X,Y),\quad X,Y\in TN.
$$ 
\item A hyperhermitian metric $g$ on $N$ is chosen so that
the fundamental two-forms $F_a = g(I_a\cdot, \cdot)$ 
are given by
$$
F_a(X,Y) 
= \frac{1}{2}(\Omega_a(X,Y) - \Omega_a(I_bX,I_bY)),\quad X,Y \in TN. 
$$
\end{enumerate} 
Here, the indices $a,b$ are so that $(a,b,c)$ is a cyclic permutation of $(1,2,3)$.
%
Let $Q_a$ be the kernel of $\theta_a$. 
Then $Q = \cap_{a=1}^3 Q_a$ is the horizontal distribution for the connection form $\theta$.
Take a triple of vertical vector fields $(T_1,T_2,T_3)$ on $M$ satisfying 
$\theta_a(T_b) = \delta_{ab}$. 
Then $Q_a$ is expressed as
$
Q_a = Q \oplus {\mathbb R}T_b \oplus {\mathbb R}T_c.
$
We define complex structures $I_a$ on $Q_a$ by
$
I_a\widetilde{X} = \widetilde{I_aX}
$
and
$
I_aT_b = T_c,
$
where $\widetilde{X}$ denotes
the horizontal lift of
a vector field $X$ on $N$.
Then it is straightforward to verify that the almost CR structures $(Q_a,I_a)$ 
are integrable, and in particular, they satisfy \eqref{integrability_condition0}, 
\eqref{integrability_condition}. 

Thus we obtain the following 

\begin{Proposition}
\label{bundleconstruction} 
Let $G$ be a Lie group of dimension three. 
Let $\pi:M \rightarrow N$ be a principal $G$-bundle over a hyperhermitian manifold $N$ 
with connection form $\theta$ and the corresponding curvature form $\Omega$, satisfying 
the conditions (i), (ii) as above. 
Then $\{(Q_a,I_a)\}_{a=1,2,3}$ defined above is a 
hyper CR structure on the total space $M$, and
$\theta$ is compatible with it. 
The Levi form $\Levi_\theta$ of $\theta$ is given by the pull-back of $g$, restricted to $Q$: 
$\Levi_\theta=(\pi^*g)|_{Q \times Q}$. 
In particular, the hyper CR structure of $M$ is strongly pseudoconvex, and together 
with $\theta$, gives a hyper pseudohermitian structure on $M$. 
\end{Proposition}

As a concrete example, the standard hyper pseudohermitian structure of the quaternionic 
Heisenberg group 
${\mathcal H}^{4n+3}$, which is an $\im \H$-bundle over the hypercomplex manifold 
${\mathbb H}^n$, can be understood by the bundle construction as above. 


An example with compact total space follows. 

\medskip\noindent
\Example\label{T^3_bundle} ($T^3$-bundle over 
$S^1 \times S^3$) \ 
Let $\widetilde{N} = \H \setminus \{0\}$ with complex coordinates $(z,w)$, 
and let $(I,J,K)$ be the standard hypercomplex structure of $\widetilde{N}$
as in \eqref{H^{n+1}_hypercomplex_structure}. 
Let $g$ be the hyperhermitian metric on $\widetilde{N}$ defined by 
$$
g = \dfrac{2(|dz|^2 + |dw|^2)}{|z|^2+|w|^2}.
$$
Then the fundamental forms $F_1, F_2, F_3$ 
are given by 
\begin{eqnarray*}
&F_1 = \dfrac{\sqrt{-1}(dz \wedge d\bar{z} + dw \wedge d\bar{w})} {|z|^2 + |w|^2}, \quad 
F_2 = \dfrac{dz \wedge dw + d\bar{z} \wedge d\bar{w}} {|z|^2 + |w|^2},& \\ 
&F_3 = \dfrac{\sqrt{-1}(d\bar{z} \wedge d\bar{w} - dz \wedge dw)} {|z|^2 + |w|^2},& 
\end{eqnarray*}
respectively. 

Let $\mu = -\log(|z|^2+|w|^2)$, a smooth function on $\widetilde{N}$, 
and define three two-forms $\Omega_a$ on $\widetilde{N}$ by $\Omega_a = d(I_ad\mu)$. 
They are given explicitly by
\begin{eqnarray*}
\Omega_1 
&=& \dfrac{2\sqrt{-1}(|w|^2dz \wedge d\bar{z} + |z|^2dw \wedge d\bar{w})} {(|z|^2+|w|^2)^2} 
- \dfrac{2\sqrt{-1}(\bar{z}wdz \wedge d\bar{w} - z\bar{w}d\bar{z} \wedge dw)} {(|z|^2+|w|^2)^2}, \\ 
\Omega_2 
&=& \dfrac{dz \wedge dw + d\bar{z} \wedge d\bar{w}}{|z|^2+|w|^2} - \dfrac{(zw-\bar{z}\bar{w})
(dz \wedge d\bar{z} - dw \wedge d\bar{w})} {(|z|^2+|w|^2)^2} \\
&& - \dfrac{(\bar{z}^2+w^2)dz \wedge d\bar{w} + (z^2+\bar{w}^2)d\bar{z} \wedge dw} {(|z|^2+|w|^2)^2},\\ 
\Omega_3 
&=& \dfrac{\sqrt{-1}(d\bar{z} \wedge d\bar{w} - dz \wedge dw)} {|z|^2+|w|^2} 
+ \dfrac{\sqrt{-1}(zw+\bar{z}\bar{w})(dz \wedge d\bar{z} - dw \wedge d\bar{w})} {(|z|^2+|w|^2)^2} \\
&& - \sqrt{-1} \dfrac{(\bar{z}^2-w^2)dz \wedge d\bar{w} - (z^2-\bar{w}^2)d\bar{z} \wedge dw} {(|z|^2+|w|^2)^2}.
\end{eqnarray*}


Let $N:= \widetilde{N}/\langle \alpha \rangle$, where $\alpha$ is a complex constant with 
$|\alpha| > 1$, acting on $\widetilde{N}$ by 
$$
\alpha{\cdot}(z,w) := ({\alpha}z,\overline{\alpha}w), \quad (z,w) \in \widetilde{N},
$$
and $\langle \alpha \rangle$ is the infinite cyclic group generated by $\alpha$. 
Then $N$ is a smooth manifold diffeomorphic to $S^1 \times S^3$, and $(g, I,J,K)$ descends 
to a hyperhermitian structure on $N$. 
Thus we obtain a hyperhermitian Hopf surface $(N,g,I,J,K)$. 
Note that though the function $\mu$ does not descend to a function on $N$, the differential $d\mu$ 
descends to a one-form on $N$, and therefore $\Omega_a$ descend to two-forms on $N$. 

Let $\pi: M=N \times T^3 \rightarrow N$ be the trivial $T^3$-bundle with fibre-coordinates 
$(t_1,t_2,t_3)$. 
(Therefore, $M$ is diffeomorphic to $S^3 \times T^4$.) 
Define an ${\mathbb R}^3$-valued one-form $\theta=(\theta_1,\theta_2,\theta_3)$ on $M$ by 
$\theta_a = dt_a + \pi^*I_ad\mu$. 
Then $\theta$ is a connection one-form in the bundle $\pi: M \rightarrow N$ with curvature form 
$\Omega = (\pi^*\Omega_1, \pi^*\Omega_2, \pi^*\Omega_3)$. 
It is straightforward to verify that the hyperhermitian Hopf surface $N$ and the forms $\Omega_a$ 
satisfy the conditions (i) and (ii) before Proposition~\ref{bundleconstruction}. 
Therefore, $M$ comes equipped with a hyper pseudohermitian structure. 
We will revisit this example in \S 3, \S 5. 

\medskip
The above construction of the hyper pseudohermitian structure 
on $S^3\times T^4$ is a special case of the following more gereral one. 
An {\em HKT manifold} is a hyperhermitian manifold $(N,g,I,J,K)$, characterized 
by the property that the fundamental forms $F_1,F_2,F_3$ satisfy $IdF_1 = JdF_2 = KdF_3$, 
where $I\omega = \omega(I \cdot,\dots,I \cdot )$ for a $k$-form $\omega$.
By a result of Banos-Swann \cite{basw}, there exists an {\em HKT-potential} $\mu$, that is,
a locally defined function $\mu$ on $N$ such that
$$
F_a = \frac{1}{2} \left( d(I_a d\mu) - I_b d(I_a d\mu)\right), 
$$
where the indices $a,b$ are as before. 
If $d(I_a d\mu)$ are globally defined on $N$ and determine integral cohomology classes of $N$, 
as in the preceding example, 
then there exists a principal $T^3$-bundle $\pi\colon M \rightarrow N$
with connection form $\theta$ 
whose curvature form $d\theta$ coincides with $( \pi^*d(I_a d\mu) )$.
Now Proposition \ref{bundleconstruction} applies, and we obtain a hyper pseudohermitian structure 
on the total space $M$.
This construction is a generalization of that due to Hernandez \cite{hern} for hyperk\"ahler manifolds. 

\section{Canonical connection}
In this section we shall construct a \qua analogue of the \TW
connection \cite{tana}, \cite{web} in CR geometry.
Throughout this section, we shall assume that the hyper and \qua
CR structures are strongly pseudoconvex.
Since our construction is modelled on that in the CR case, we first
review it briefly.
\par
Let $(M,\theta)$ be a pseudohermitian manifold. 
(Recall that the underlying almost CR structure is assumed to be partially integrable.) 
As in \cite{sta}, let $T$ be an {\em arbitrary} transverse vector field 
such that $\theta(T)=1$; except for this point,
we follow the explanation of the \TW connection due to Rumin \cite{rum}, where $T$ is the Reeb field
from the beginning. 
For each $k>0$, define a Riemannian metric $g_{M,k}$
on $M$ by
$$
g_{M,k} = g + k\theta^2,
$$
where $g$ is extended to a positive semidefinite form on $TM$ by defining
$g(T,\cdot) = 0$.
There is a unique connection $\nabla$ which satisfies
$\nabla g_{M,k} = 0$ for all $k$ and among such connections, has as small torsion as possible.
It is characterized by the following conditions:
\begin{enumerate}
\renewcommand{\theenumi}{\roman{enumi}}
\renewcommand{\labelenumi}{(\theenumi)}
\item the subbundle $Q$ is preserved by $\nabla$;
\item $g$ and $T$ are $\nabla$-parallel;
\item the torsion tensor $\Tor$ of $\nabla$ satisfies
\begin{enumerate}
\renewcommand{\theenumi}{\alph{enumi}}
\renewcommand{\labelenumi}{(\theenumi)}
\item $\Tor(X,Y)_Q = 0$,\quad $X$, $Y\in Q$;
\item $X\in Q \mapsto \Tor(T,X)_Q \in Q$ is $g$-symmetric,
\end{enumerate}
\end{enumerate}
where $E_Q$ denotes the $Q$-component of a tangent vector $E$
with respect to the splitting $TM = Q \oplus \R T$. 
It follows from $\nabla g=0$ and (iii-a) that for $X,Y\in \Gamma(Q)$,
$\nabla_X Y$ is given by
\begin{eqnarray}\label{tw1}
2 g(\nabla_X Y,Z) &=& X g(Y,Z) + Y g(X,Z) - Z g(X,Y)
+ g([X,Y]_Q,Z)\\
&& - g([X,Z]_Q,Y) - g([Y,Z]_Q,X)\quad
\mbox{for all $Z\in \Gamma(Q)$}. \nonumber
\end{eqnarray}
\par
We now determine $T$ so that the corresponding connection $\nabla$
be as close to being a unitary connection as possible.
Since any orthogonal connection on a hermitian line bundle is unitary,
this step does not work when $n=1$.
Hence we assume $n\geq 2$ hereafter.
Fix an arbitrary $T$, and write $\widehat{T} = T + 2JV$ for
$V\in \Gamma(Q)$.
Then by (\ref{tw1}), the corresponding connections $\nabla$ and
$\widehat{\nabla}$ are related by
\begin{equation}\label{tw2}
\widehat{\nabla}_X Y = \nabla_X Y + g(JX,Y)JV - g(JV,Y)JX - g(JV,X)JY.
\end{equation}
Let $\{e_1,\dots,e_n\}$ be a local unitary basis for $Q^{1,0}$,
and write
$$
\nabla e_i = \sum_{j=1}^n \left(\omega_{i\bar{j}}e_j + \omega_{ij}
\overline{e_j}\right),\quad
\widehat{\nabla} e_i = \sum_{j=1}^n \left(\widehat{\omega}_{i\bar{j}}e_j
+ \widehat{\omega}_{ij}\overline{e_j}\right).
$$
Note that $\widehat{\nabla}$ is a unitary connection if and only if
$\widehat{\omega}_{ij} = 0$ for all $i,j$.
Using (\ref{tw2}) we obtain
$$
\widehat{\omega}_{ij}(e_k) = \omega_{ij}(e_k),\quad 
\widehat{\omega}_{ij}(\overline{e_k}) = \omega_{ij}(\overline{e_k})
+ \delta_{kj}\overline{V_i} - \delta_{ki}\overline{V_j},
$$
where we write $V = \sum_{i=1}^n (V_i e_i + \overline{V_i}
\overline{e_i} )$.
$\widehat{\omega}_{ij}(e_k)$ are independent of $V$, and they all vanish
if and only if the underlying almost CR structure is integrable, while $\widehat{\omega}_{ij}
(\overline{e_k})$
can always be made zero by an appropriate choice of $V$.
Indeed, using (\ref{tw1}) and $d(d\theta) (e_i, e_j, \overline{e_k}) = 0$,
we obtain
\begin{equation}\label{tw4}
\omega_{ij}(\overline{e_k}) = \frac{1}{2}(\delta_{ki}d\theta(T, e_j)
- \delta_{kj}d\theta(T, e_i)).
\end{equation}
Hence, by choosing
\begin{equation}\label{tw3}
\overline{V_i} = \frac{1}{2} d\theta(T, e_i), \quad i=1,\dots,n,
\end{equation}
we can achieve $\widehat{\omega}_{ij}(\overline{e_k}) = 0$.
Note that (\ref{tw3}) is equivalent to $\widehat{T}$ being the Reeb
field associated with $\theta$.
In particular, the resulting connection $\widehat{\nabla}$ is the
\TW connection, as generalized to the partially integrable case by
Tanno \cite{tann}.

We now turn to the \qua case, and first treat the hyper CR case.
So let $(M,\theta)$ be a strongly pseudoconvex hyper pseudohermitian manifold. 
As above, our construction proceeds in two steps:
first, for each choice of an admissible three-plane field $Q^\perp$, we construct a certain 
uniquely determined connection $\nabla$ on $TM$. 
Next we determine $Q^\perp$ so that, when restricted to a connection on $Q$, $\nabla$ be ``as close to 
an $Sp(n)\cdot Sp(1)$-connection as possible." 

Let $g$ denote the Levi form of $\theta$; it is a metric on $Q$. 
Let $Q^\perp$ be an admissible three-plane field, so that we have the splitting 
\begin{equation}\label{splitting}
TM = Q \oplus Q^\perp. 
\end{equation}
Set $g^\perp:={\theta_1}^2 + {\theta_2}^2 + {\theta_3}^2$, and denote its restriction to $Q^\perp$ by 
the same symbol. 
We define a family of Riemannian metrics $g_{M, k}$ on $M$ by $g_{M,k} = g + k g^\perp$, 
where $k>0$ and $g$ is extended to a positive semidefinite form on $TM$ by defining $g(U,\cdot) = 0$ 
for all $U\in Q^\perp$. 
The splitting \eqref{splitting} is orthogonal with respect
to all $g_{M, k}$.
As in the CR case, there is a unique connection $\nabla$ which satisfies $\nabla g_{M,k} = 0\quad 
\mbox{for all $k$}$ and among such connections, has as small torsion as possible.
We state a characterization of this connection as

\begin{Proposition}\label{noncanconn1}
Let $(M,\theta)$ be a 
hyper pseudohermitian manifold, 
and let $Q^\perp$ be an admissible three-plane field. 
Then there exists a unique connection $\nabla$ on $TM$ satisfying  the following conditions:
\begin{enumerate}
\renewcommand{\theenumi}{\roman{enumi}}
\renewcommand{\labelenumi}{(\theenumi)}
\item the subbundles $Q$ and $Q^\perp$ are preserved by $\nabla$; 
\item $g$ and $g^\perp$ are $\nabla$-parallel; 
\item for $X, Y \in Q$ and $U,V \in Q^\perp$, 
\begin{enumerate}
\renewcommand{\theenumi}{\alph{enumi}}
\renewcommand{\labelenumi}{(\theenumi)}
\item $\Tor(X,Y)_Q = 0$;
\item $\Tor(U,V)_{Q^\perp} = 0$; 
\item $X \in Q \mapsto \Tor(U,X)_Q \in Q$ is $g$-symmetric;
\item $U \in Q^\perp \mapsto \Tor(U,X)_{Q^\perp} \in Q^\perp$ 
is $g^\perp$-symmetric, 
\end{enumerate}
\end{enumerate}
where $E_Q$ and $E_{Q^\perp}$ respectively denote the $Q$- and $Q^\perp$-components 
of a tangent vector $E$ with respect to the splitting \eqref{splitting}. 
\end{Proposition}

\begin{proof}
%
Throughout the proof, let $X,Y,Z\in \Gamma(Q)$ and $U,V,W\in \Gamma(Q^\perp)$. 
Suppose that $\nabla$ is a connection on $TM$ satisfying the conditions 
stated in the proposition. 
Then the conditions (i), (ii), (iii-a), (iii-b) force $\nabla$ to satisfy 
\begin{eqnarray}\label{eq209}
2 g(\nabla_X Y,Z) &=& X g(Y,Z) + Y g(X,Z) - Z g(X,Y) + g([X,Y]_Q,Z)
\nonumber\\
&& - g([X,Z]_Q,Y) - g([Y,Z]_Q,X)
\end{eqnarray}
and
\begin{eqnarray}\label{eq210}
2 g^\perp(\nabla_U V,W) &=& U g^\perp(V,W) + V g^\perp(U,W) - W g^\perp(U,V) 
+ g^\perp([U,V]_{Q^\perp},W)
\nonumber\\
&& - g^\perp([U,W]_{Q^\perp},V) - g^\perp([V,W]_{Q^\perp},U).
\end{eqnarray} 
Since ${\rm Tor}(U,X)_Q = \nabla_U X - [U,X]_Q$, the condition $\nabla g = 0$ implies that 
$$
g(\mbox{\rm Tor}(U,X)_Q,Y) + g(X,\mbox{\rm Tor}(U,Y)_Q) 
= Ug(X,Y) - g([U,X]_Q,Y) - g(X,[U,Y]_Q).
$$
Therefore, the condition (iii-c) determines ${\rm Tor}(U,X)_Q$ by 
$$
g(\mbox{\rm Tor}(U,X)_Q,Y) = \frac{1}{2}(Ug(X,Y) - g([U,X]_Q,Y) - g(X,[U,Y]_Q)), 
$$
and this gives 
\begin{eqnarray}\label{eq211}
g(\nabla_U X,Y) = 
\frac{1}{2}(Ug(X,Y) + g([U,X]_Q,Y) - g(X,[U,Y]_Q)). 
\end{eqnarray}
Similarly, the condition (iii-d) determines ${\rm Tor}(U,X)_{Q^\perp}$ by  
$$
g^\perp(\mbox{\rm Tor}(U,X)_{Q^\perp},V) 
= -\frac{1}{2}(Xg^\perp(U,V) + g^\perp([U,X]_{Q^\perp},V) + g^\perp(U,[V,X]_{Q^\perp})), 
$$
which gives 
\begin{equation}\label{eq212} 
g^\perp(\nabla_X U,V) = 
\frac{1}{2}(Xg^\perp(U,V) - g^\perp ([U,X]_{Q^\perp},V) + g^\perp (U,[V,X]_{Q^\perp})). 
\end{equation}
Conversely, (\ref{eq209}),  (\ref{eq210}), (\ref{eq211}), (\ref{eq212}) determine 
a connection $\nabla$ on $TM$ uniquely, and $\nabla$ satisfies the conditions stated in the proposition. 
\end{proof}

\medskip\noindent
\Remark One can generalize Proposition \ref{noncanconn1} to a quaternionic pseudohermitian structure 
in an obvious manner.

\medskip
Our next task is to determine $Q^\perp$. 
As in the CR case, we shall work with complex frames.
Let
$
\onezero = \{ X\in Q\otimes {\Bbb C} \mid IX = \rmo X \}.
$
Then we have the decomposition
$
Q\otimes {\Bbb C} = \onezero \oplus \overline{\onezero},
$
orthogonal with respect to the Levi form $g$ regarded as
a hermitian form on $Q\otimes {\Bbb C}$.
Take a local orthonormal frame $\{ \varepsilon_1,\dots,\varepsilon_{4n} \}$ 
for $Q$ satisfying
\begin{equation}\label{rframe}
\varepsilon_{4k-2} = I \varepsilon_{4k-3},\,\,
\varepsilon_{4k-1} = J \varepsilon_{4k-3},\,\,
\varepsilon_{4k} = K \varepsilon_{4k-3}
(= I \varepsilon_{4k-1}),\,\, k=1,\dots, n.
\end{equation}
(Such a local orthonormal frame for $Q$ is said to be {\em adapted}.) 
Then
\begin{equation}\label{cframe}
\left\{ e_{2k-1}= \left(\varepsilon_{4k-3}-\rmo\varepsilon_{4k-2}\right)/\sqrt{2},\,\,
e_{2k}= \left(\varepsilon_{4k-1}-
\rmo\varepsilon_{4k}\right)/\sqrt{2} \right\}_{1\leq k\leq n}
\end{equation} 
is a local unitary frame for $\onezero$, and $J,K : \onezero
\rightarrow \overline{\onezero}$ are given by
$$
Je_{2k-1} = \overline{e_{2k}},\,\, 
Je_{2k} = - \overline{e_{2k-1}};\,\,
Ke_{2k-1} = -\rmo \overline{e_{2k}},\,\,
Ke_{2k} = \rmo \overline{e_{2k-1}}.
$$

Choose an admissible three-plane field $Q^\perp$, and let $\nabla$ be the corresponding connection 
as in Proposition \ref{noncanconn1}.
Regarding $\nabla$ as a connection on $Q$, let $\omega$ be the matrix of connection forms 
with respect to the above local frame; its components are given by 
$$
\nabla e_i = \sum_{j=1}^{2n} \left(\omega_{i\bar{j}} e_j + \omega_{ij}
\overline{e_j}\right). 
$$
Since $\nabla g = 0$, we have 
$\omega_{i\bar{j}} = -\overline{\omega_{j\bar{i}}}$
and $\omega_{ij} = - \omega_{ji}$, 
namely, $(\omega_{i\bar{j}})$ is skew-hermitian and
$(\omega_{ij})$ is skew-symmetric.
Note that $\nabla\overline{e_i} = \overline{\nabla e_i}$
since $\nabla$ is a real connection.

%

We now further restrict the domain of $\nabla$ by considering 
the $Q$-partial connection 
$$
\nabla^Q : (X,Y) \in \Gamma(Q) \times \Gamma(Q) \mapsto \nabla_X Y 
\in \Gamma(Q). 
$$
In other words, we regard $\omega$ as being defined on $Q$. 
Let $(sp(n) + sp(1))^\perp$ denote the orthogonal complement of $sp(n) + sp(1)$ 
in $so(4n)$ with respect to the Killing inner product. 
Then $\omega$'s $(sp(n) + sp(1))^\perp$-component $\omega^{\rm obs}$ gives 
an obstruction for the $Q$-partial connection $\nabla^Q$ 
to preserve the $Sp(n)\cdot Sp(1)$-structure of $Q$, and $\omega^{\rm obs}$ 
is small if and only if $\nabla^Q $ is close to being an $Sp(n)\cdot Sp(1)$-partial connection. 
$\omega^{\rm obs}$ is tensorial, and $(\omega^{\rm obs})_q$ is an element of
${Q_q}^* \otimes (sp(n) + sp(1))^\perp$ for each point $q$.

Note that when $n=1$, since $Sp(1)\cdot Sp(1) = SO(4)$, the $SO(4)$-connection $\nabla$ 
necessarily preserves the $Sp(1)\cdot Sp(1)$-structure of $Q$, and therefore the obstruction 
tensor $\omega^{\rm obs}$ vanishes irrespective of the choice of $Q^\perp$. 
Hence we assume $n\geq 2$ hereafter, and use reprensentation theory to make the requirement that 
$\omega^{\rm obs}$ be small more explicit. 
${Q_q}^* \otimes (sp(n) + sp(1))^\perp$ is an $Sp(n)\times Sp(1)$-module 
whose model is $\H^n \otimes (sp(n) + sp(1))^\perp$.
Swann \cite{swa} wrote down the irreducible decomposition of this module
explicitly, which we shall review.
Let $E$ (resp. $H$) be the standard complex $Sp(n)$(resp. $Sp(1)$)-module, 
with the left action of $Sp(n)$ (resp. $Sp(1)$) through the inclusion
$Sp(n) \subset SU(2n)$ (resp. $Sp(1)=SU(2)$).
Then we have
\begin{eqnarray}\label{hn}
&\H^n \otimes \C \cong E \otimes H&\\
\label{spnsp1}
&(sp(n) + sp(1))^\perp \otimes \C \cong \Lambda^2_0E \otimes S^2H&
\end{eqnarray}
as complex $Sp(n)\times Sp(1)$-modules.
In fact, 
\begin{eqnarray}\label{congseq}
so(4n) \otimes \C &\cong& \Lambda^2 \H^n \otimes \C\quad
\mbox{(as $SO(4n)$-modules)}\nonumber\\
&\cong& \Lambda^2 (E \otimes H)\\
&\cong& S^2E \oplus S^2H
\oplus (\Lambda^2_0E \otimes S^2H). \nonumber
\end{eqnarray}
Since $S^2E \cong sp(n)$ and $S^2H \cong sp(1)$, we conclude
(\ref{spnsp1}).
It follows from (\ref{hn}) and (\ref{spnsp1}) that
\begin{eqnarray}\label{congseq2}
(\H^n \otimes (sp(n) + sp(1))^\perp) \otimes \C
&\cong& (E\otimes H) \otimes (\Lambda^2_0E \otimes S^2H)\nonumber\\
&\cong& (E\otimes \Lambda^2_0E) \otimes (H\otimes S^2H)\\
&\cong& (K\oplus \Lambda^3_0E\oplus E) \otimes (S^3H\oplus H),\nonumber
\end{eqnarray}
where $K$ is the irreducible complex $Sp(n)$-module with highest weight
$(2,1,0,\dots,0)$.
Therefore, we have the irreducible decomposition
\begin{eqnarray}\label{irred_decomp}
\lefteqn{(\H^n \otimes (sp(n) + sp(1))^\perp) \otimes \C} \nonumber \\
&\cong& (K\otimes S^3H) \oplus (\Lambda^3_0E\otimes S^3H) \oplus 
(E\otimes S^3H)\\
&& \oplus (K\otimes H) \oplus (\Lambda^3_0E\otimes H) \oplus (E\otimes H). \nonumber 
\end{eqnarray}
It can be shown that all the components of $\omega^{\rm obs}$ other than the one corresponding 
to $E\otimes H$ are stable under a change of $Q^\perp$.  
It is also not possible in general to remove the component $\omega^{E\otimes H}$ 
of $\omega^{\rm obs}$ in $E\otimes H$. 
As we shall prove in Theorem \ref{cantriple0} below, $\omega^{E\otimes H}$ can be removed 
by a suitable choice of $Q^\perp$ if one assumes a stronger condition than strong pseudoconvexity 
which is called ultra-pseudoconvexity and will be defined below. 

To proceed, we shall first express the condition 
\begin{equation}\label{require0}
\omega^{E\otimes H} = 0 
\end{equation} 
in a more explicit form.

\begin{Lemma}\label{require_rewrite}
The condition \eqref{require0} is rewritten as follows:\,\, for $l=1,\dots,n$, 
\begin{eqnarray}\label{require2.1}
&& \sum_{k=1}^n \biggl\{ \frac{1}{2} (\overline{\omega_{2k-1,\overline{2l-1}}} - \omega_{2k,\overline{2l}}) (e_{2k-1})
+ \omega_{2k-1,2l-1} (\overline{e_{2k-1}}) \nonumber \\ 
&&\qquad +\, \frac{1}{2} (\omega_{2k-1,\overline{2l}} + \overline{\omega_{2k,\overline{2l-1}}})(e_{2k})
+ \omega_{2k, 2l-1} (\overline{e_{2k}}) \\ 
&&\qquad +\, \frac{1}{2n}(\omega_{2k-1,\overline{2k-1}} + \omega_{2k,\overline{2k}}) (e_{2l-1}) 
+ \frac{1}{n} \omega_{2k-1,2k} (\overline{e_{2l}}) \biggr\} = 0, \nonumber
\end{eqnarray}
and
\begin{eqnarray}\label{require2.2}
&& \sum_{k=1}^n \biggl\{ \frac{1}{2} (\overline{\omega_{2k-1,\overline{2l}}} + \omega_{2k,\overline{2l-1}}) (e_{2k-1})
+ \omega_{2k-1,2l} (\overline{e_{2k-1}}) \nonumber \\ 
&&\qquad +\, \frac{1}{2} (-\, \omega_{2k-1,\overline{2l-1}} + \overline{\omega_{2k,\overline{2l}}})(e_{2k})
+ \omega_{2k, 2l} (\overline{e_{2k}}) \\ 
&&\qquad +\, \frac{1}{2n}(\omega_{2k-1,\overline{2k-1}} + \omega_{2k,\overline{2k}}) (e_{2l}) 
- \frac{1}{n} \omega_{2k-1,2k} (\overline{e_{2l-1}}) \biggr\} = 0. \nonumber
\end{eqnarray}
\end{Lemma}

We shall postpone the proof of this lemma to the next section. 

\medskip\noindent
\Definition
We say that a hyper CR structure is {\em ultra-pseudoconvex} if the symmetric bilinear form $h$ on the subbundle $Q$ 
defined by
$$
h(X,Y) = (2n+4) \Levi_\theta (X,Y) - \sum_{a=1}^3 d\theta_a(X,I_a Y),\quad X,Y\in Q,
$$
is positive or negative definite for some (hence any) compatible, $\R^3$-valued one-form $\theta$.

\medskip 
Since the component of $h$ invariant under $I,J,K$ is $(2n+1) g$, a hyper CR structure is \spc 
if it is ultra-pseudoconvex. 
The sphere $S^{4n+3}$ and the quaternionic Heisenberg group $\mathcal{H}^{4n+3}$ are ultra-pseudoconvex, 
since on these hyper CR manifolds, the three complex Levi forms $d\theta_a(\cdot,I_a \cdot)$ coincide 
on $Q$ and therefore they are equal to $\Levi_\theta$. 
It is easy to see that strictly convex real hypersurfaces in ${\Bbb H}^{n+1}$ are ultra-pseudoconvex.  
In particular, the ellipsoids as in \S 2 are ultra-pseudoconvex.
We give a less obvious example. 

\medskip\noindent
\Example\label{T^3_bundle_usc} Let $(M,\theta)$ be the hyper pseudohermitian manifold as in Example \ref{T^3_bundle}. 
Recall that $M$ is the total space of the (trivial) $T^3$-bundle $\pi\colon M\rightarrow N$ over 
the hyperhermitian Hopf surface $(N,g,I,J,K)$. 
We have 
$$
\mathrm{Levi}_\theta = \pi^* g = \frac{2(|dz|^2 + |dw|^2)}{|z|^2+|w|^2} 
$$
and 
\begin{eqnarray}\label{T^3_bundle_h}
h &=& 6\, \mathrm{Levi}_\theta - \sum^3_{a=1}d\theta_a(\cdot,I_a\cdot) 
\nonumber \\
&=& \dfrac{4(|dz|^2 + |dw|^2)}{|z|^2+|w|^2} + \dfrac{2(\bar{z}dz + zd\bar{z} + \bar{w}dw 
+ wd\bar{w})^2}{(|z|^2+|w|^2)^2} \\ 
&=& 2\, \mathrm{Levi}_\theta + 2 (d\mu)^2. \nonumber 
\end{eqnarray}
Therefore, the hyper CR structure of $M$ is ultra-pseudoconvex. 

\medskip\noindent
\Example\label{deformation_qheisen_usc} 
For the hyper pseudohermitian structure as in Example \ref{deformation_qheisen}, we have 
\begin{eqnarray*}
h &=& \frac{1}{2}\sum_\alpha \biggl\{ \Bigl( (n+2)\Lambda_\alpha - B^1_\alpha - C^2_\alpha 
- D^3_\alpha - \sum_{a=1}^3 A^a_\alpha \Bigr) \left(dx^0_\alpha\right)^2 \\ 
&& \phantom{\frac{1}{2}\sum_\alpha \left\{ \right.} 
+ \Bigl( (n+2)\Lambda_\alpha - A^1_\alpha - D^2_\alpha - C^3_\alpha - \sum_{a=1}^3 B^a_\alpha \Bigr) 
\left(dx^1_\alpha\right)^2\\
&& \phantom{\frac{1}{2}\sum_\alpha \left\{ \right.} 
+ \Bigl( (n+2)\Lambda_\alpha - D^1_\alpha - A^2_\alpha - B^3_\alpha - \sum_{a=1}^3 C^a_\alpha \Bigr) 
\left(dx^2_\alpha\right)^2\\
&& \phantom{\frac{1}{2}\sum_\alpha \left\{ \right.} 
+ \Bigl( (n+2)\Lambda_\alpha - C^1_\alpha - B^2_\alpha - A^3_\alpha - \sum_{a=1}^3 D^a_\alpha \Bigr) 
\left(dx^3_\alpha\right)^2 \biggr\}. 
\end{eqnarray*} 
Note that $h$ can be degenerate or even indefinite according to various choices of 
$A^a_\alpha,\cdots, D^a_\alpha$ (e.g.~if $A^a_\alpha = n+1$ and $B^a_\alpha = C^a_\alpha = D^a_\alpha = -n/3$, 
then $h$ is indefinite). 
Thus the hyper CR structure may not be ultra-pseudoconvex, even though it is strongly pseudoconvex. 

\medskip 
Note that $h$ remains unchanged under the deformation of hyper CR structure and $\theta$ as in (\ref{qcr}), 
\eqref{theta_deform}. 
Thus the definition of ultra-pseudoconvexity extends to the \qua CR structure.

\begin{Theorem}\label{cantriple0}
Let $(M,\theta)$ be an ultra-pseudoconvex hyper pseudohermitian manifold of dimension $> 7$. 
Then there exists a unique admissible three-plane field $Q^\perp$ such that 
the corresponding connection $\nabla$ as in Proposition \ref{noncanconn1} 
satisfies \eqref{require0}. 
\end{Theorem}

We call $Q^\perp$ of the theorem the {\em canonical three-plane field}, and the corresponding 
admissible triple $(T_1, T_2,T_3)$ the {\em canonical triple}. 
The corresponding connection, denoted by $D$, is a quaternionic analogue of the Tanaka-Webster connection 
in complex CR geometry. 
We call it the {\em canonical connection} associated with $(M,\theta)$. 

\medskip\noindent
{\em Proof of Theorem \ref{cantriple0}.}
Fix an arbitrary admissible triple $(T_1,T_2,T_3)$ of reference, and let $\nabla$ be the corresponding 
connection given by Proposition \ref{noncanconn1}. 
Let $\widehat{T}_a = T_a + 2I_a V$ for $V\in \Gamma(Q)$. 
We will show that $V$ can be chosen uniquely so that $\widehat{\nabla}$, the connection 
corresponding to $(\widehat{T}_1, \widehat{T}_2, \widehat{T}_3)$, satisfies \eqref{require0}. 

First, we have for $X,Y\in Q$,  
\begin{equation}\label{bracket_change}
[X,Y]_Q^{\,\,\widehat{\mbox{}}} = [X,Y]_Q + 2 \sum_{a=1}^3 d\theta_a(X,Y)I_aV,
\end{equation}
where $[X,Y]_Q^{\,\,\widehat{\mbox{}}}$ is the $Q$-component of $[X,Y]$ with
respect to $(\widehat{T}_1, \widehat{T}_2,\widehat{T}_3)$. 
Using (\ref{bracket_change}) and (\ref{eq209}) with $(Y, Z) = (e_I, e_J)$, we obtain
\begin{eqnarray}\label{difference_omega}
\widehat{\omega}_{IJ}(X) &=& \omega_{IJ}(X) + \sum^3_{a = 1}d\theta_a(X,e_I)g(I_aV,e_J) \nonumber \\
&&- \sum^3_{a=1}d\theta_a(X,e_J)g(I_aV,e_I) - \sum^3_{a=1}d\theta_a(e_I,e_J)g(I_aV,X).
\end{eqnarray}
Here, $\omega_{IJ}$ (resp. $\widehat{\omega}_{IJ}$) are connection forms of $\nabla$ (resp. $\widehat{\nabla}$), 
and the indices $I, J$ range over $1,2,\dots,2n,\overline{1},\overline{2},\dots,\overline{2n}$. 

Let $\omega_{2l-1}$ (resp. $\widehat{\omega}_{2l-1}$) denote the left-hand side of \eqref{require2.1}, 
computed for $\nabla$ (resp. $\widehat{\nabla}$). 
Then by using \eqref{difference_omega}, we obtain 
\begin{eqnarray}
\lefteqn{\widehat{\omega}_{2l-1} - \omega_{2l-1}} \nonumber \\ 
&=& \sum^3_{a=1} \biggl\{ - \left( \frac{3}{2}+\frac{1}{2n} \right) d\theta_a(I_aV,e_{2l-1}) \nonumber \\
&&\phantom{\sum^3_{a=1} \biggl\{} - \left( \frac{1}{2} + \frac{1}{2n} \right) \sum^n_{k=1}
\bigl(d\theta_a(e_{2k-1},\overline{e_{2k-1}}) + d\theta_a(e_{2k},\overline{e_{2k}}) \bigr) g(I_aV,e_{2l-1}) \\
&&\phantom{\sum^3_{a=1} \biggl\{} - \left( 1+\frac{1}{n} \right) \sum^n_{k=1} d\theta_a(e_{2k-1},e_{2k})g(I_aV,\overline{e_{2l}}) 
\nonumber \\ 
&&\phantom{\sum^3_{a=1} \biggl\{} + \frac{1}{n} \sum^n_{k=1} \bigl[ \bigl(
d\theta_a(\overline{e_{2k}},e_{2l-1}) - d\theta_a(e_{2k-1},\overline{e_{2l}}) \bigr) g(I_aV,e_{2k}) 
\nonumber \\
&&\phantom{\sum^3_{a=1} \biggl\{ + \frac{1}{n} \sum^n_{k=1} \bigl[} 
+ \bigl( d\theta_a(\overline{e_{2k-1}},e_{2l-1}) + d\theta_a(e_{2k},\overline{e_{2l}})\bigr) 
g(I_aV,e_{2k-1}) \bigr] \biggr\}. \nonumber 
\end{eqnarray}

The right-hand side is simplified as follows. 
Since $d\theta_a$ is $I_a$-invariant, the sum on the second line vanishes if $a=2,3$. 
We compute the sum for $a=1$:  
\begin{eqnarray} 
\lefteqn{\sum^n_{k=1} \bigl( d\theta_1(e_{2k-1},\overline{e_{2k-1}}) 
+ d\theta_1(e_{2k},\overline{e_{2k}}) \bigr) g(I_1V,e_{2l-1})} \nonumber\\
&=&
- \sum^n_{k=1}
\bigl( d\theta_1(e_{2k-1},\overline{e_{2k-1}}) 
+ d\theta_1(e_{2k},\overline{e_{2k}}) \bigr)g(V,{I_1}e_{2l-1}) \nonumber\\
&=&
\sum^n_{k=1}
\bigl( d\theta_1(e_{2k-1},I_1\overline{e_{2k-1}}) 
+ d\theta_1(\overline{e_{2k}},I_1e_{2k}) \bigr) g(V,e_{2l-1}) \\
&=&
\sum^n_{k=1}2g(e_{2k-1},\overline{e_{2k-1}})g(V,e_{2l-1}) \nonumber\\
&=&
2n\, g(V,e_{2l-1}). \nonumber
\end{eqnarray}

The sum on the third line vanishes if $a=1$; for $a=2$, we compute: 
\begin{eqnarray} 
\sum^n_{k=1} d\theta_2(e_{2k-1},e_{2k})g(I_2V,\overline{e_{2l}}) 
&=& \sum^n_{k=1}d\theta_2(e_{2k-1},I_2 \overline{e_{2k-1}}) g(V,e_{2l-1}) \nonumber \\ 
&=& \sum^n_{k=1} g(e_{2k-1}, \overline{e_{2k-1}}) g(V,e_{2l-1}) \\ 
&=& n\, g(V,e_{2l-1}). \nonumber 
\end{eqnarray} 
Likewise, for $a=3$, we obtain 
\begin{equation}
\sum^n_{k=1} d\theta_3(e_{2k-1},e_{2k})g(I_3V,\overline{e_{2l}}) = n\, g(V,e_{2l-1}).
\end{equation} 

We compute the sum on the fourth and fifth lines: 
\begin{eqnarray*}
&&\sum^n_{k=1} \bigl[ \bigl(
d\theta_a(\overline{e_{2k}},e_{2l-1}) - d\theta_a(e_{2k-1},\overline{e_{2l}}) \bigr) g(I_aV,e_{2k}) \\
&&\phantom{\sum^n_{k=1} \bigl[} 
+ \bigl( d\theta_a(\overline{e_{2k-1}},e_{2l-1}) + d\theta_a(e_{2k},\overline{e_{2l}})\bigr) 
g(I_aV,e_{2k-1}) \bigr] \\ 
&=& \sum^{2n}_{i=1} \bigl[ \bigl(
d\theta_a(\overline{e_{i}},e_{2l-1}) + d\theta_a(I_2 \overline{e_{i}},\overline{e_{2l}}) \bigr) g(I_aV,e_{i}) \\
&&\phantom{\sum^n_{k=1} \bigl[}
 + \bigl(
d\theta_a(e_{i},e_{2l-1}) + d\theta_a(I_2 e_{i},\overline{e_{2l}}) \bigr) g(I_aV,\overline{e_{i}})  \bigr] \\
&& - \sum^{2n}_{i=1} \bigl(
d\theta_a(e_{i},e_{2l-1}) + d\theta_a(I_2 e_{i},\overline{e_{2l}}) \bigr) g(I_aV,\overline{e_{i}}) \\
&=& d\theta_a(I_a V,e_{2l-1}) + d\theta_a(I_2 I_a V,\overline{e_{2l}}) \\ 
&& - \sum^{2n}_{i=1} \bigl(
d\theta_a(e_{i},e_{2l-1}) + d\theta_a(I_2 e_{i},\overline{e_{2l}}) \bigr) g(I_aV,\overline{e_{i}}). 
\end{eqnarray*}
We find that the right-hand side is equal to $2 [ g(V, e_{2l-1}) - d\theta_a(V, I_a e_{2l-1}) ]$ for all $a$.
Indeed, if $a=1$, the sum on the second line vanishes, and 
\begin{eqnarray*} 
d\theta_1(I_2 I_1 V,\overline{e_{2l}}) &=& d\theta_1(I_2 V, I_1 I_2 e_{2l-1})\\ 
&=& ( d\theta_1(I_2 V, I_1 I_2 e_{2l-1})  + d\theta_1(V, I_1 e_{2l-1}) ) - d\theta_1(V, I_1 e_{2l-1}) \\ 
&=& 2g(V, e_{2l-1}) - d\theta_1(V, I_1 e_{2l-1}). 
\end{eqnarray*} 
If $a=2$, the sum on the first line is equal to $2 d\theta_2(I_2 V, e_{2l-1})$, and the sum on 
the second line becomes 
$$
- 2 \sum^{2n}_{i=1} d\theta_2(e_{i}, I_2\overline{e_{2l}})  g(I_2V,\overline{e_{i}}) 
= - 2\, g(I_2V,\overline{e_{2l}}) = - 2\, g(V, e_{2l-1}), 
$$ 
since $d\theta_2(e_{i}, I_2\overline{e_{2l}}) = g(e_{i}, \overline{e_{2l}}) = \delta_{i, 2l}$. 
The $a=3$ case is similar. 

We thus conclude that 
\begin{eqnarray} 
\widehat{\omega}_{2l-1} - \omega_{2l-1} 
&=& \left(-3n - 3 + \frac{6}{n} \right) g(V, e_{2l-1}) + \left( \frac{3}{2} - \frac{3}{2n} \right) \sum_{a=1}^3 
d\theta_a(V, I_a e_{2l-1}) \nonumber \\ 
&=& - \frac{3n-3}{2n} h(V, e_{2l-1}). \nonumber 
\end{eqnarray}

We obtain a similar result for the left-hand side $\omega_{2l}$ of \eqref{require2.2}, and conclude that 
$\widehat{\omega}^{E\otimes H}$ vanishes if and only if 
$$
\frac{3n-3}{2n} h(V, e_{2l-1}) = \omega_{2l-1}\quad \mbox{and}\quad \frac{3n-3}{2n} h(V, e_{2l}) = \omega_{2l}
$$ 
for $l=1,\cdots, n$. 
By ultra-pseudoconvexity, there exists a unique $V$ which satisfies this last system of linear equations.
\hfill $\square$

\medskip
We now extend the above construction of the canonical connection for a hyper pseudohermitian structure 
to that for a \qua pseudohermitian structure. 
To do this, it suffices to verify that given a quaternionic pseudohermitian structure, 
the condition \eqref{require0} over each local hyper pseudohermitian structure is actually global. 
Let $(M,\Theta)$ be an ultra-pseudoconvex quaternionic pseudohermitian manifold 
of dimension $> 7$. 
There is a global $Sp(n)\cdot Sp(1)$-bundle $Q$ over $M$, and  
let $\mathcal{P}$ be the bundle of frames for $Q$ which are adapted (cf.~\eqref{rframe}) with respect to 
some triple $(I_a)$ of compatible complex structures; this is a principal $Sp(n)\cdot Sp(1)$-bundle over $M$. 
Let $\nabla$ be any 
$SO(4n)$-connection on $Q$, and for any local section 
$\varepsilon = (\varepsilon_1,\cdots, \varepsilon_{4n})$ of $\mathcal{P}$, let $\omega$ be the corresponding 
matrix of (real-valued) connection forms, given by $\nabla \varepsilon = \varepsilon\otimes \omega$. 
Note that this $\omega$ is essentially the same as the previous one (when $\nabla$ is the connection 
given by Proposition \ref{noncanconn1});  
as a collection of local matrix-valued forms, the present $\omega$ is the expression of the previous one in terms of 
real frames \eqref{rframe} rather than complex ones \eqref{cframe}. 
As before, we regard $\omega$ as being defined on $Q$. 
If $\varepsilon$ changes as $\varepsilon\mapsto \varepsilon a$, where $a$ is a local $Sp(n)\cdot Sp(1)$-valued 
function, then $\omega$ transforms as $\omega\mapsto a^{-1}\omega a + a^{-1} da$. 

Now let $\iota$ be the standard representation of $Sp(n)\cdot Sp(1)$ on $\H^n$, $\iota^*$ its dual, and let 
$\mathrm{Ad}$ be the adjoint representation of $SO(4n)$ on its Lie algebra $so(4n)$. 
By restriction, the last representation induces one of $Sp(n)\cdot Sp(1)$ on $(sp(n)+sp(1))^\perp$, 
which we denote by the same symbol. 
We then consider the representation $\iota^*\otimes \mathrm{Ad}$ of $Sp(n)\cdot Sp(1)$ on 
${\H^n}^*\otimes (sp(n)+sp(1))^\perp$, and construct the vector bundle 
$$
\mathcal{E} = \mathcal{P}\times_{\iota^*\otimes \mathrm{Ad}} ({\H^n}^*\otimes (sp(n)+sp(1))^\perp) 
= Q^*\otimes \mathcal{P}\times_{\mathrm{Ad}} (sp(n)+sp(1))^\perp. 
$$ 
Let $\omega^{\rm obs}$ be the $(sp(n)+sp(1))^\perp$-component of $\omega$. 
Then the above transformation law for $\omega$ ensures that the local forms $\omega^{\rm obs}$ give 
a global section of $\mathcal{E}$. 
According to the irreducible decomposition \eqref{irred_decomp}, the bundle $\mathcal{E}\otimes \C$ splits 
and $\omega^{\rm obs}$ thereby decomposes, both globally. 
Therefore, $\omega^{E\otimes H}$, the $E\otimes H$-component of $\omega^{\rm obs}$, is also global. 

By the obsevation we just made, we obtain the following conclusion. 

\begin{Theorem}\label{gluing}
Let $(M, \Theta=\{\theta_U\})$ be an ultra-pseudoconvex \qua pseudohermitian manifold 
of dimension $> 7$. 
Then the local canonical three-plane fields $\{(Q^\perp)_U\}$ and the local canonical connections 
$\{D_U\}$ for local hyper pseudohermitian structures 
patch together to give a global admissible three-plane field $Q^\perp$ and a global connection $D$, respectively. 
\end{Theorem}

\noindent
\Definition
Let $(M, \Theta)$ be an ultra-pseudoconvex \qua pseudohermitian manifold 
of dimension $> 7$. 
We call $Q^\perp$ and $D$ of Theorem \ref{gluing} the {\em canonical three-plane field}
and the {\em canonical connection}, respectively, associated with $\Theta$. 

\medskip
We now derive, for future use, the transformation law 
for the canonical triple under a conformal change of 
(hyper) pseudohermitian structure. 
\begin{Proposition}\label{conformal_change}
Let $(M,\theta)$ be an ultra-pseudoconvex hyper pseudohermitian manifold 
of dimension $> 7$. 
Let $\theta^\prime = e^{2f}\theta$, and $(T_1,T_2,T_3)$ (resp. $(T_1^\prime, T_2^\prime, T_3^\prime)$) 
the canonical triple corresponding to $\theta$ (resp. $\theta^\prime$). 
Then we have 
$$
T_a^\prime = e^{-2f}(T_a + 2I_a W), 
$$ 
where $W\in \Gamma(Q)$ is uniquely determined by  
\begin{equation}\label{cantriple6}
h(W,X) = -(2n+1)d_bf(X),\quad X\in Q. 
\end{equation}
\end{Proposition}

\begin{proof}
We regard $W$ as the unknown and verify that it must satisfy \eqref{cantriple6}. 
Set $\alpha = d_b f$, $g = \Levi_\theta$ and $g' = \Levi_{\theta'} = e^{2f} g$. 
Let $D$ (resp. $D'$) be the canonical connection for $\theta$ (resp. $\theta'$). 
As in the proof of Theorem \ref{cantriple0}, we obtain for $X\in Q$, 
\begin{eqnarray}\label{cantriple4}
\omega_{IJ}^\prime (X) &=& \omega_{IJ}(X) + \alpha(X)g(e_I,e_J) + \alpha(e_I)g(X,e_J) - \alpha(e_J) g(X,e_I)\nonumber\\ 
&& + \sum_{a=1}^3 d\theta_a(X,e_I) g(I_aW, e_J) - \sum_{a=1}^3 d\theta_a(X,e_J) g(I_aW, e_I)\\ 
&& - \sum_{a=1}^3 d\theta_a(e_I,e_J) g(I_aW, X),\nonumber 
\end{eqnarray}
where $\omega_{IJ}$ (resp. $\omega_{IJ}^\prime$) are connection forms 
of $D$ (resp. $D'$), $\{ e_1,\cdots, e_{2n} \}$ is a $g$-unitary frame as in \eqref{cframe} and 
the indices $I, J$ range over $1,2,\dots,2n,\overline{1},\overline{2},\dots,\overline{2n}$.
It should be also noted that $\omega_{IJ}^\prime$ are computed with respect to the $g'$-unitary 
frame $\{ e^{-f} e_i \}$. 
Since the last three terms on the right-hand side appear in \eqref{difference_omega}, we can use 
the computation in the proof of Theorem \ref{cantriple0}. 
Denoting the left-hand sides of \eqref{require2.1}, \eqref{require2.2} by $\omega_{2l-1}$, $\omega_{2l}$, 
respectively, we obtain 
$$
e^f \omega_i^\prime - \omega_i = - \frac{3n-3}{2n} \left\{ h(e_i, W) + (2n+1)\alpha(e_i) \right\}
$$
for $i=1,\cdots, 2n$. 
Again, note that $\omega_i^\prime$ are computed with respect to $\{ e^{-f} e_i \}$. 
Since the left-hand sides of these identities vanish, we must have 
$
h(e_i, W) = - (2n+1)\alpha(e_i)
$
for $i=1,\cdots, 2n$. 
This completes the proof of Proposition \ref{conformal_change}.
\end{proof}

\noindent 
\Remark\label{canonical_triple_S^{4n+3}} 
For the sphere $S^{4n+3}$, we have $h = (2n+1) \Levi_\theta$. 
Therefore, \eqref{cantriple6} gives $W = - d_b f^\#$, which is consistent with
the transformation law \eqref{transform2}.

\medskip
We conclude this section with some comments on the curvature 
of the canonical connection. 
Let $(M, \Theta)$ be an ultra-pseudoconvex \qua pseudohermitian manifold 
of dimension $> 7$, 
and $D$ the associated canonical connection. 
Let $R$ and $\Ric$ denote the curvature and Ricci tensors of $D$, respectively.
For $X,Y\in Q$, we have
$
\Ric(X,Y) = \sum_{i=1}^{4n} g(R(\varepsilon_i,X)Y,\varepsilon_i),
$
where $\{\varepsilon_1,\dots,\varepsilon_{4n}\}$ is an orthonormal basis for $Q$ 
with respect to the Levi form $g = \Levi_\theta$. 
The {\em pseudohermitian Ricci tensor} $r$ is the component of $\Ric|_Q$
(restriction to $Q$) which is symmetric and invariant under $I,J, K$.
The {\em pseudohermitian scalar curvature} is $s = \tr_g (\Ric|_Q) = \sum_{i=1}^{4n} \Ric(\varepsilon_i, \varepsilon_i)$. 

Let $\theta_S$ and $\theta_{H}$ be the standard pseudohermitian structures 
of the sphere $S^{4n+3}$ and the quaternionic Heisenberg group ${\mathcal H}^{4n+3}$, 
respectively. 
Recall from \S 2 that they are related by $\theta_S = e^{2f}\sigma \theta_{H} \sigma^{-1}$ 
for some real-valued function $f$ and $Sp(1)$-valued function $\sigma$. 

The curvature of $\theta_H$ 
vanishes identically, and the curvature of $\theta_S$ coincides with 
that of $e^{2f} \theta_{H}$. 
There are formulas computing the curvature of 
the pseudohermitian structure of the form $e^{2f} \theta_H$, 
and by using them, we obtain $r_{\theta_S} = 2(n+2)\Levi_{\theta_S}$ and 
$s_{\theta_S} = 8n(n+2)$. 

In a future work, we shall study the curvature of quaternionic 
pseudohermitian manifold in detail. 

\section{Proof of Lemma \ref{require_rewrite}}

Let $\omega$, $\omega^{\rm obs}$ and $\omega^{E\otimes H}$ be the forms as in the previous section. 
Recall that we regard them as being defined on $Q$. 
Then we have 

\begin{Lemma}\label{omega_eh} 
The coefficients of $\omega^{E\otimes H}$ corresponding to a 
standard basis of $E\otimes H$ are given by 
the left-hand sides of \eqref{require2.1} and \eqref{require2.2} in Lemma \ref{require_rewrite}
with $l=1,\dots,n$, and their complex conjugates.

\end{Lemma}

The rest of this section is devoted to the proof of Lemma \ref{omega_eh}. 

To prove Lemma \ref{omega_eh}, we start by making the correspondences (\ref{hn}) and (\ref{spnsp1}) more explicit.
For (\ref{hn}), let $I: {\Bbb H}^n \rightarrow {\Bbb H}^n$ be the complex
structure given by the right multiplication of $\mbox{\bf i}^{-1}$,
and set $V = \{ X \in {\Bbb H}^n \otimes \C \mid IX = \sqrt{-1}X \}$, 
so that we have ${\Bbb H}^n \otimes {\Bbb C} = V \oplus \overline{V}$.
Let $(\varepsilon_1,\dots,\varepsilon_{4n})$ be the standard basis for
${\Bbb H}^n = {\Bbb R}^{4n}$, and define a complex basis for $V$ by
$e_{2k-1}=(\varepsilon_{4k-3}-\sqrt{-1}\varepsilon_{4k-2})/\sqrt{2}$, 
$e_{2k}=(\varepsilon_{4k-1}-\sqrt{-1}\varepsilon_{4k})/\sqrt{2}$ ($k = 1,\dots,n$). 
Also, let $({\bf e}_1,\dots, {\bf e}_{2n})$ and $({\bf f}_1, {\bf f}_2)$
respectively denote the standard basis for $E=\C^{2n}$ and $H=\C^2$.
Then the correspondence
$$
e_{2k-1}\,\, \leftrightarrow\,\, {\bf e}_{2k-1} \otimes {\bf f}_2,
\,\, e_{2k}\,\, \leftrightarrow\,\, {\bf e}_{2k} \otimes {\bf f}_2, \,\, 
\overline{e_{2k-1}}\,\, \leftrightarrow\,\, -{\bf e}_{2k} \otimes
{\bf f}_1,\,\, \overline{e_{2k}}\,\, \leftrightarrow\,\,
{\bf e}_{2k-1} \otimes {\bf f}_1
$$
($k = 1,\dots, n$) gives an isomorphism
${\Bbb H}^n \otimes {\Bbb C} \cong E \otimes H$.
For (\ref{spnsp1}), we can find the elements of $(sp(n) + sp(1))^\perp \otimes \C$ 
corresponding to generators of $\Lambda^2_0 E \otimes S^2 H$,
by tracing the isomorphisms in (\ref{congseq}) backwards: 
\begin{eqnarray*}
({\bf e}_{2k-1}\wedge {\bf e}_{2l-1})\otimes ({\bf f}_1\cdot {\bf f}_1)
&\leftrightarrow& \overline{e_{2k}}\wedge \overline{e_{2l}},\\ 
({\bf e}_{2k-1}\wedge {\bf e}_{2l-1})\otimes ({\bf f}_2\cdot {\bf f}_2)
&\leftrightarrow& e_{2k-1} \wedge e_{2l-1},\\
({\bf e}_{2k-1}\wedge {\bf e}_{2l-1})\otimes ({\bf f}_1\cdot {\bf f}_2)
&\leftrightarrow& \frac{1}{2}
(e_{2k-1}\wedge \overline{e_{2l}} - e_{2l-1}\wedge \overline{e_{2k}}),
\end{eqnarray*}
\begin{eqnarray*}
({\bf e}_{2k}\wedge {\bf e}_{2l}) \otimes ({\bf f}_1\cdot {\bf f}_1)
&\leftrightarrow& \overline{e_{2k-1}}\wedge \overline{e_{2l-1}},\\
({\bf e}_{2k}\wedge {\bf e}_{2l})\otimes ({\bf f}_2\cdot {\bf f}_2)
&\leftrightarrow& e_{2k} \wedge e_{2l}, \\
({\bf e}_{2k}\wedge {\bf e}_{2l})\otimes ({\bf f}_1\cdot {\bf f}_2)
&\leftrightarrow& \frac{1}{2}
(- \overline{e_{2k-1}} \wedge e_{2l} + \overline{e_{2l-1}} \wedge e_{2k}),
\end{eqnarray*}
\begin{eqnarray*}
({\bf e}_{2k}\wedge {\bf e}_{2l-1})_0\otimes ({\bf f}_1\cdot {\bf f}_1)
&\leftrightarrow& - \overline{e_{2k-1}}\wedge \overline{e_{2l}}
+ \frac{1}{n} \delta_{kl} \sum_{m=1}^n \overline{e_{2m-1}}\wedge
\overline{e_{2m}},\\
({\bf e}_{2k}\wedge {\bf e}_{2l-1})_0\otimes ({\bf f}_2\cdot {\bf f}_2)
&\leftrightarrow& e_{2k} \wedge e_{2l-1}
+ \frac{1}{n} \delta_{kl} \sum_{m=1}^n e_{2m-1} \wedge e_{2m},\\
({\bf e}_{2k}\wedge {\bf e}_{2l-1})_0\otimes ({\bf f}_1\cdot {\bf f}_2)
&\leftrightarrow& \frac{1}{2} (- \overline{e_{2k-1}} \wedge e_{2l-1}
- \overline{e_{2l}} \wedge e_{2k})\\
&&\mbox{} + \frac{1}{2n} \delta_{kl} \sum_{m=1}^n
(\overline{e_{2m-1}}\wedge e_{2m-1} + \overline{e_{2m}}\wedge e_{2m}). 
\end{eqnarray*}

The isomorphism $S^3 H \oplus H \cong H \otimes S^2 H$ embeds $H$ into $H \otimes S^2 H$ by
$$
s_H : w\in H \,\,\mapsto\,\,  {\bf f}_1 \otimes ({\bf f}_2 \cdot w)
- {\bf f}_2 \otimes ({\bf f}_1 \cdot w) \in H \otimes S^2 H. 
$$
Likewise, the isomorphism $K \oplus \Lambda^3_0 E \oplus E\cong E \otimes \Lambda^2_0 E$ 
embeds $E$ into $E \otimes \Lambda^2_0 E$ by
$$s_E : w \in E \,\,\mapsto\,\, \sum_{k=1}^n \left[
{\bf e}_{2k-1} \otimes ({\bf e}_{2k} \wedge w)_0
- {\bf e}_{2k} \otimes (\mbox{\bf e}_{2k-1} \wedge w)_0 \right]
\in E \otimes \Lambda^2_0 E.
$$

We are now ready to give 

\medskip\noindent
{\em Proof of Lemma  \ref{omega_eh}.}\quad 
We shall identify the coefficient of $\omega^{E\otimes H}$ corresponding to ${\bf e}_{2l-1}\otimes {\bf f}_1
\in E\otimes H$, to which corresponds the following element of 
$(\H^n \otimes \Lambda^2 \H^n) \otimes \C$: 
\begin{eqnarray*}
&&\sum_{k=1}^n \biggl\{
\overline{e_{2k}}\otimes \biggl[ \frac{1}{2}
(- \overline{e_{2k-1}}\wedge e_{2l-1} - \overline{e_{2l}}\wedge e_{2k})
+ \frac{1}{2n} \delta_{kl} \sum_{m=1}^n
(\overline{e_{2m-1}}\wedge e_{2m-1} + \overline{e_{2m}}\wedge e_{2m})\biggr]\\
&&\qquad + \overline{e_{2k-1}}\otimes \frac{1}{2} (e_{2k-1}\wedge \overline{e_{2l}} - e_{2l-1}\wedge \overline{e_{2k}}) 
 - e_{2k-1}\otimes \biggl[ - \overline{e_{2k-1}}\wedge \overline{e_{2l}}
+ \frac{1}{n} \delta_{kl} \sum_{m=1}^n \overline{e_{2m-1}}\wedge \overline{e_{2m}}\biggr]\\
&&\qquad
+ e_{2k}\otimes (\overline{e_{2k}}\wedge \overline{e_{2l}})\biggr\}. 
\end{eqnarray*}
Note that at each point $q \in M$, we can express $\omega$ as
\begin{eqnarray*}
\omega = \sum_{1 \leq i < j \leq 2n}\omega_{ij}{\otimes}\varphi_i \wedge \varphi_j
+ \sum_{1 \leq i < j \leq 2n}\omega_{\bar{i}\bar{j}} {\otimes}\overline{\varphi_i} \wedge \overline{\varphi_j}
+ \sum^{2n}_{i,j=1}\omega_{i\bar{j}}{\otimes}\varphi_i \wedge \overline{\varphi_j},
\end{eqnarray*}
where $(\varphi_1,\dots,\varphi_{2n})$ is the dual of the unitary basis $(e_1,\dots, e_{2n})$ for $(Q_q)^{1,0}\cong V$. 
Then the coefficient of $\omega^{E\otimes H}$ corresponding to ${\bf e}_{2l-1}\otimes {\bf f}_1$ is given by 
\begin{eqnarray*}
&& \sum_{k=1}^n \biggl\{\frac{1}{2} (-\, \omega_{\overline{2k-1},2l-1} - \omega_{\overline{2l},2k})(\overline{e_{2k}})
+ \frac{1}{2n}(\omega_{\overline{2k-1},2k-1} + \omega_{\overline{2k},2k}) (\overline{e_{2l}}) \\
&&\qquad +\, \frac{1}{2} (\omega_{2k-1,\overline{2l}} - \omega_{2l-1,\overline{2k}}) (\overline{e_{2k-1}})
+ \omega_{\overline{2k-1},\overline{2l}} (e_{2k-1}) - \frac{1}{n} \omega_{\overline{2k-1},\overline{2k}} (e_{2l-1})
\\
&&\qquad 
+\, \omega_{\overline{2k},\overline{2l}} (e_{2k}) \biggr\},
\end{eqnarray*}
which is the complex conjugate of the left-hand side of \eqref{require2.2}. 
Likewise, computing the coefficients corresponding to the other basis elements of $E\otimes H$,
we obtain the left-hand sides of \eqref{require2.1}, their complex conjugates and those of \eqref{require2.2}. 
\hfill$\square$

\section{Comparison to quaternionic contact structure} 

As mentioned in the introduction, some quaternionic analogues of CR structures 
other than those in this paper have been studied by several authors 
(cf.\,\cite{alkami0}, \cite{alkami}, \cite{biq}, \cite{hern}). 

In this section, we first review the definition of quaternionic contact structure, 
introduced by Biquard \cite{biq}, and the canonical connection, called the Biquard connection,  
associated with a choice of metric. 
We then compare the quaternionic CR structure to the quaternionic contact structure. 
We observe that while a quaternionic contact structure can always 
be ``extended'' to a quaternionic CR structure, the quaternionic contact structure is 
more restrictive than the quaternionic CR structure. 

\medskip\noindent
\Definition\label{qc_structure}
A {\em quaternionic contact structure} 
on a $(4n+3)$-dimensional manifold $M$ 
is a corank three bundle $Q$  
equipped with a $CSp(n)\cdot Sp(1)$-structure satisfying a compatibility 
condition. 
That is, we have a conformal class $[\gamma]$ of metrics on $Q$ and a two-sphere bundle $\mathbb{I}$ 
over $M$ of 
complex structures $I\colon Q\rightarrow Q$, $I^2 = -\id$, and these satisfy the following conditions: 
\begin{enumerate}
\renewcommand{\theenumi}{\roman{enumi}}
\renewcommand{\labelenumi}{(\theenumi)}
\item $\gamma(IX, IY) = \gamma(X, Y)$ for all $I\in \mathbb{I}$ and $X,Y\in Q$.
\item $\mathbb{I}$ locally admits 
sections $I_a$, $a=1,2,3$, satisfying the quaternion relations $I_1I_2 = - I_2I_1 = I_3$ and 
$\mathbb{I} = \{ v_1I_1 + v_2I_2 + v_3I_3 \mid {v_1}^2 + {v_2}^2 + {v_3}^2 = 1 \}$. 
\item $Q$ is locally the kernel of an $\R^3$-valued one-form $\eta = (\eta_1, \eta_2, \eta_3)$ 
satisfying the compatibility relations 
\begin{equation}\label{compatibility_relation1}
\gamma(I_aX, Y) = d\eta_a(X, Y),\quad a=1,2,3, 
\end{equation}
where $X,Y\in Q$.
\end{enumerate}

\medskip
Note that \eqref{compatibility_relation1} is equivalent to $\gamma(X,Y) = d\eta_a(X, I_aY)$; 
in particular, \eqref{compatibility_relation1} implies the condition (i). 
Note also that if \eqref{compatibility_relation1} holds, then for any other triple 
$(I_1', I_2', I_3')$ of complex structures as in the condition (ii), there exists 
an $\R^3$-valued one-form $\eta' = (\eta_1', \eta_2', \eta_3')$ so that the compatibility relations 
$\gamma(I_a'X, Y) = d\eta_a'(X, Y)$ hold. 
Indeed, if $I_a' = \sum_{p=1}^3 s_{ap} I_p$, where $(s_{ap})$ is an $SO(3)$-valued function, 
then it suffices to choose $\eta_a' = \sum_{p=1}^3 s_{ap} \eta_p$. 
(Actually, this is a unique choice of $\eta_a'$, as verified by argument similar to that 
in the proof of Proposition \ref{QC_more_restrictive} below.) 

On a quaternionic contact manifold of 
dimension $> 7$ with a choice of metric $\gamma$ on $Q$ in the conformal 
class, Biquard constructed a canonical connection $D^B$, called the {\em Biquard connection} 
 (cf.~\cite{duc1} for the seven-dimensional case). 
He also gave a distinguished rank three subbundle $Q^\perp$ of $TM$ 
complementary to $Q$. 
The connection $D^B$ and our canonical connection $D$ on a quaternionic 
pseudohermitian manifold are similar but differ in some respects: 
first, $D^B$ preserves the $Sp(n)\cdot Sp(1)$-structure of $Q$, while 
$D$ does not in general, because of the generality of our structure; 
second, the torsion tensor $\Tor$ of $D^B$, restricted to $Q\times Q^\perp$, 
has no $Q^\perp$-component and is more sensitive to the 
$GL(n, \mathbb{H})\cdot Sp(1)$-structure of $Q$, 
because of the ``quaternionic extension'' used in the construction of $D^B$. 

The bundle $Q^\perp$ can be explicitly described. 
Choose a local $\R^3$-valued one-form $(\eta_1, \eta_2, \eta_3)$ as in 
Definition \ref{qc_structure}. 
Then $Q^\perp$ is locally generated by vector fields 
$\{R_a\}_{a=1,2,3}$ characterized by 
\begin{equation}\label{Reeb_condition}
\eta_a(R_b) =\delta_{ab},\quad d\eta_a(R_a, X) = 0,\quad X\in Q, 
\end{equation}
and they further satisfy
\begin{equation*}\label{stronger_identity}
d\eta_b(R_a,X) = - d\eta_a(R_b, X),\quad X\in Q.
\end{equation*}

Let $M$ be a quaternionic contact manifold, with the associated corank three subbundle $Q$ of $TM$ 
and two-sphere bundle $\mathbb{I}$ of complex structures of $Q$.
Then there are quaternionic CR structures having $(Q, \mathbb{I})$ as the underlying structure. 
To show this, fix a metric $\gamma$ on $Q$ and choose an arbitrary rank three bundle $Q^\perp$ 
transverse to $Q$. 
First, we construct a local hyper CR structure. 
So choose $(I_1,I_2,I_3)$ locally and then choose a local $\R^3$-valued 
one-form $(\eta_1,\eta_2,\eta_3)$ so that \eqref{compatibility_relation1} holds.  
Since $\eta_a|_{Q^\perp}$ form a local coframe for $Q^\perp$, there is a unique triple 
$(T_1, T_2, T_3)$ of local sections of $Q^\perp$ such that $\eta_a(T_b) = \delta_{ab}$.
Then set $Q_a = Q\oplus \R T_b\oplus \R T_c$, 
and extend $I_a\colon Q\rightarrow Q$ 
to $I_a\colon Q_a\rightarrow Q_a$ by defining $I_aT_b = T_c$ and $I_aT_c = -T_b$. 
Note that we have $\ker \eta_a = Q_a$ and $\eta_a\circ I_b = \eta_c$, 
and $\{ (Q_a, I_a) \}_{a=1,2,3}$ satisfies the conditions for an almost hyper CR structure. 
Moreover, it is integrable. 
Indeed, for $X,Y\in \Gamma(Q)$, we compute using \eqref{compatibility_relation1}: 
\begin{eqnarray*}
\eta_a([X,Y] - [I_aX, I_aY]) &=& -d\eta_a(X,Y) + d\eta_a(I_aX,I_aY)\\ 
&=& - \gamma(I_aX,Y) - \gamma(X,I_aY)\\ 
&=& 0,
\end{eqnarray*} 
\begin{eqnarray*}
\lefteqn{\eta_b( I_a([X,Y] - [I_aX, I_aY]) - ([X,I_aY]+[I_aX,Y]) )}\\  
&=& -\eta_c([X,Y] - [I_aX, I_aY]) - \eta_b([X,I_aY]+[I_aX,Y])\\ 
&=& d\eta_c(X,Y) - d\eta_c(I_aX, I_aY) + d\eta_b(X,I_aY) + d\eta_b(I_aX,Y)\\ 
&=& \gamma(I_cX, Y) - \gamma(I_cI_aX, I_aY) + \gamma(I_bX, I_aY) + \gamma(I_bI_aX,Y)\\ 
&=& 0, 
\end{eqnarray*}
and likewise, 
$$
\eta_c( I_a([X,Y] - [I_aX, I_aY]) - ([X,I_aY]+[I_aX,Y]) ) = 0.
$$
In this way, for each local choice of $(I_1,I_2,I_3)$, we have the corresponding local 
hyper CR structure $\{ (Q_a, I_a) \}$. 
We now verify that two such local hyper CR structures $\{ (Q_a, I_a) \}$ and $\{ (Q_a', I_a') \}$
satisfy the gluing condition \eqref{qcr} if they overlap. 
Suppose that $I_a' = \sum_p s_{ap} I_p$ as endomorphisms of $Q$, where $(s_{ap})$ is 
an $SO(3)$-valued function. 
Choose $\eta_a'=\sum_p s_{ap} \eta_p$ and $T_1', T_2', T_3'$ be the corresponding 
local sections of $Q^\perp$. 
We must show that 
\begin{equation}\label{gluing_condition}
T_a' = \sum_p s_{ap} T_p\quad \mbox{and}\quad I_a' = \biggl( \sum_p s_{ap} 
\widetilde{I_p} \biggr) \biggm|_{Q_a'},
\end{equation} 
where the notation $\widetilde{I_p}$ is as in \S 1. 
Since the former relations are clear, 
it sufficces to verify the latter relations of \eqref{gluing_condition}. 
Set $I_a'' = \sum_p s_{ap} \widetilde{I_p}$. 
We compute  
$$
I_a''T_b' = \sum_{p,q} s_{ap} s_{bq} \widetilde{I_p} T_q, 
$$ 
and restricting the indices $p,q$ to those which extends to a cyclic permutation $(p,q,r)$ of $(1,2,3)$, 
we further rewrite the right-hand side as 
$$
\sum ( s_{ap} s_{bq} - s_{aq} s_{bp} ) \widetilde{I_p} T_q = \sum_r s_{cr} T_r = T_c'. 
$$ 
Thus $I_a'' = I_a'$, which gives the latter relations of \eqref{gluing_condition}.  

The above construction actually gives a quaternionic pseudohermitian structure 
such that the associated Levi form is the metric $\gamma$, and therefore, we have 
the canonical three-plane field $(Q^\perp)'$. 
While $(Q^\perp)'$ differs from $Q^\perp$ in general, $(Q^\perp)' = Q^\perp$ 
holds when $Q^\perp$ is Biquard's one, locally generated by the vector fields 
$\{R_a\}_{a=1,2,3}$ satisfying the Reeb condition 
\eqref{Reeb_condition}. 
We record this fact as the following 

\begin{Proposition}\label{cannonical=reeb} 
Let $M$ be a quaternionic contact manifold of dimension $> 7$ 
with a choice of metric $\gamma$ on the corank three bundle $Q$.  
Let $Q^\perp$ be the rank three bundle locally generated 
by the vector fields $\{R_a\}_{a=1,2,3}$ satisfying \eqref{Reeb_condition},  
and equip $M$ with a quaternionic pseudohermitian structure in the way 
as above. 
Then $Q^\perp$ gives the canonical three-plane field associated with 
the quaternionic pseudohermitian structure. 
\end{Proposition}

\begin{proof}
Let $D^B$ be the Biquard connection associated with the metric $\gamma$, 
and $\nabla$ the affine connection, given by Proposition \ref{noncanconn1}, 
associated with the quaternionic 
pseudohermitian structure and the admissible three-plane field $Q^\perp$.  
When regarded as $Q$-partial connections, they coincide with each other, since 
they are characterized by the same condition \eqref{eq209}. 
We know that $D^B$ restricts to an $Sp(n)\cdot Sp(1)$-connection on $Q$, 
and therefore $\nabla$ restricts to a $Q$-partial connection preserving the 
$Sp(n)\cdot Sp(1)$ strucutre of $Q$. 
So the obstruction $\omega^{\rm obs}$ vanishes, and in particular, 
$\omega^{E\otimes H} = 0$. 
This means that $Q^\perp$ is the canonical three-plane field (and 
$\nabla$ is the canonical connection) associated with 
the quaternionic pseudohermitian structure. 
We are done.
\end{proof}

The following proposition generalizes \cite[Proposition 2.1]{duc2}, which characterizes a quaternionic 
contact real hypersurface in a quaternionic manifold, to an arbitrary quaternionic CR manifold.  
It shows that the above mentioned coincidence of the Levi form with 
the metric $\gamma$ is actually the case for {\em any} quaternionic CR structure which has 
$(Q, \mathbb{I})$ as the underlying structure. 

\begin{Proposition}\label{QC_more_restrictive}
Let $M$ be a quaternionic contact manifold, with the associated corank three subbundle $Q$ of $TM$, 
two-sphere bundle $\mathbb{I}$ of complex structures of $Q$ and conformal class $[\gamma]$ of metrics on $Q$. 
Then for any quaternionic CR structure on $M$ having $(Q,\mathbb{I})$ as the underlying structure, 
the conformal class consisting of all Levi forms coincides with $[\gamma]$. 
Furthermore, if $\theta=(\theta_1, \theta_2, \theta_3)$ is a local $\R^3$-valued one-form 
compatible with 
a local hyper CR structure $\{(Q_a, I_a)\}_{a=1,2,3}$ (constituting the quaternionic CR structure), 
then $\Levi_\theta$, $(I_1,I_2,I_3)$ and $\theta$ satisfy the compatibility relations 
\begin{equation}\label{compatibility_relation2}
\Levi_\theta(X,Y) = d\theta_a (X, I_a Y),\quad a=1,2,3, 
\end{equation}
where $X,Y\in Q$. 
In particular, the quaternionic CR structure under consideration must be ultra-pseudoconvex.
\end{Proposition}

\begin{proof} 
Let $(I_1,I_2,I_3)$ be as in the statement of the proposition. 
The condition (iii) says that there exists a local $\R^3$-valued one-form 
$\eta = (\eta_1, \eta_2, \eta_3)$ such that the kernel of $\eta$ coincides with $Q$ 
and \eqref{compatibility_relation1} holds: $\gamma(I_aX, Y) = d\eta_a(X, Y)$.  
First observe that $\theta_a$ may be expressed as $\theta_a = \sum_{p=1}^3 s_{ap}\eta_p$, 
where $(s_{ap})$ is a $GL(3,\R)$-valued function. 
Then for $X,Y\in Q$, 
$$
d\theta_a(X,Y) = \sum_p s_{ap}\, d\eta_p(X,Y) = \sum_p s_{ap}\, \gamma(I_pX,Y) = \gamma(J_aX, Y), 
$$
where we set $J_a = \sum_p s_{ap} I_p$. 
We compute $d\theta_a(X,I_aY)$ in two ways: 
$$
d\theta_a(X,I_aY) = \gamma(J_aX, I_aY) = - \gamma(I_aJ_aX, Y)
$$
and 
$$
d\theta_a(X,I_aY) = - d\theta_a(I_aX,Y) = - \gamma(J_aI_aX,Y).
$$
This implies $I_aJ_a = J_aI_a$ as endomorphisms of $Q$, and therefore $J_a$ is a multiple of 
$I_a$ by a scalar-valued function: $J_a = \lambda_a I_a$, $\lambda_a\neq 0$.  
Now $d\theta_a(X,I_aY) = \lambda_a \gamma(X,Y)$, and so $\Levi_\theta(X,Y) = \lambda_a \gamma(X,Y)$. 
Therefore, $\lambda_1=\lambda_2=\lambda_3$, and denoting this function by $\lambda$, 
we have $d\theta_a(X,I_aY) = \Levi_\theta(X,Y) = \lambda \gamma(X,Y)$. 
\end{proof}

We now look at a real hypersurface $M$ in a quaternionic manifold, 
and compare the quaternionic CR structure to the quaternionic contact structure in this case. 
As explained in \S 2, $M$ has a canonical quaternionic CR structure, and therefore there exists 
a canonical corank three subbundle $Q$ of $TM$, together with a canonical two-sphere bundle 
$\mathbb{I}$ of complex structures of $Q$ as in (ii) of the above definition. 
In contrast, a real hypersurface in a quaternionic manifold does not admit in general a quaternionic contact structure 
which the canonical $(Q, \mathbb{I})$ underlies. 
Ellipsoids as in \S 2 supply concrete examples; an ellipsoid in $\mathbb{H}^{n+1}$ does not 
admit a quaternionic contact structure having the canonical $(Q, \mathbb{I})$ as the underlying structure, 
unless the ellipsoid is a quaternionic one. 
This follows from Proposition \ref{QC_more_restrictive}. 
Indeed, for an ellipsoid which is not quaternionic, we observed in \S 2 that the complex Levi forms 
$\Levi_{\theta_a} = d\theta_a(\cdot,I_a\cdot)$ do not coincide on $Q$ for the standard choice of $\theta$. 
In particular, \eqref{compatibility_relation2} cannot hold. 
Note that, since any ellipsoid in $\mathbb{H}^{n+1}$ is diffeomorphic to 
the sphere $S^{4n+3}$, 
it does admit a quaternionic contact structure by pulling back that of the sphere.
However, the underlying structure $(Q,\mathbb{I})$ is different from the canonical one of the ellipsoid. 

The hyper CR manifolds of Example \ref{deformation_qheisen} and Example \ref{T^3_bundle} give 
concrete examples of intrinsic quaternionic CR manifold 
which does not admit a quaternionic contact structure with the same 
underlying structure $(Q, \mathbb{I})$. 
Indeed, for the pseudohermitian structure $\theta$ of Example \ref{deformation_qheisen}, 
\eqref{compatibility_relation2} holds if and only if
$$
A^1_\alpha+B^1_\alpha =  C^1_\alpha+D^1_\alpha = A^2_\alpha+C^2_\alpha =  B^2_\alpha+D^2_\alpha
= A^3_\alpha+D^3_\alpha =  B^3_\alpha+C^3_\alpha = \frac{\Lambda_\alpha}{2}
$$
for all $\alpha$. 
In other words, unless this last condition is satisfied, \eqref{compatibility_relation2} cannot hold. 
For $\theta$ of Example \ref{T^3_bundle}, \eqref{T^3_bundle_h} shows that $h$ and $\Levi_\theta$ 
are not proportional, and therefore \eqref{compatibility_relation2} cannot hold. 

\section{Appendix.}

\subsection{All $(Q_{\bf v},I_{\bf v})$ satisfy \eqref{integrability_condition0} and \eqref{integrability_condition}}
\label{ap_integrability}

Let $M$ be a hyper CR manifold, and fix a function ${\bf v} = (v_1, v_2, v_3)$ with values in $S^2 \subset \R^3$.
In this subsection, we will show that $(Q_{\bf v},I_{\bf v})$ satisfies the conditions 
\eqref{integrability_condition0} and \eqref{integrability_condition} for all $X,Y\in \Gamma(Q)$. 

Let $\theta = (\theta_1, \theta_2, \theta_3)$ be an $\R^3$-valued one-form on $M$ 
compatible with the hyper CR structure. 
Set $\theta_{\bf v} := v_1\theta_1 + v_2\theta_2 + v_3\theta_3$, 
so that $Q_{\bf v} = \ker \theta_{\bf v}$. 
Fix an admissible triple $(T_1,T_2,T_3)$ such that $\theta_a(T_a) = 1$. 
Recall from \S 1 that endomorphisms $\widetilde{I_a}$ of $TM$ are defined by setting 
$\widetilde{I_a}X = {I_a}X$ for $X \in Q_a$ and $\widetilde{I_a}T_a = 0$, and that 
$I_{\bf v} = (v_1\widetilde{I_1} + v_2\widetilde{I_2} + v_3\widetilde{I_3})|_{Q_{\bf v}}$. 
We will use the following relations:
\begin{equation}\label{theta_ai_1}
\theta_c = \theta_a{\circ}\widetilde{I}_b = - \theta_b{\circ}\widetilde{I}_a, \quad
\theta_a{\circ}\widetilde{I}_a = 0, 
\end{equation}
\begin{equation}\label{theta_ai_2}
\widetilde{I_c} = \left\{
\begin{array}{rl}
\widetilde{I_a}{I_b} & \mbox{\rm on } \ Q_b,\\
-\widetilde{I_b}{I_a} & \mbox{\rm on } \ Q_a. 
\end{array}
\right.
\end{equation}

\begin{Proposition} \label{familyintegrable}
For any $S^2$-valued function ${\bf v}$, 
$(Q_{\bf v},I_{\bf v})$ satisfies 
\begin{eqnarray}
&[ X,Y ] - [I_{\bf v}X,I_{\bf v}Y ] \in \Gamma(Q_{\bf v}), \label{1stint}\\
&I_{\mathbf v} ( [X,Y] - [ I_{\bf v}X,I_{\bf v}Y ] ) - [ X,I_{\bf v}Y ] - [ I_{\bf v}X,Y ] )\in \Gamma(Q) 
\label{2ndint}
\end{eqnarray}
for all $X, Y \in \Gamma(Q)$.
\end{Proposition}

\begin{proof}  
We first prove \eqref{1stint} by verifying $\theta_{\bf v}([ X,I_{\bf v}Y ] + [ I_{\bf v}X,Y ]) = 0$
for $X, Y \in \Gamma(Q)$. 
Plug \eqref{integrability_condition} into $\theta_c$, use \eqref{eq101} and replace $Y$ by $I_cY$. 
We then obtain 
$$
\theta_b([X,I_cY] + [I_aX,I_bY]) + \theta_c([X,I_bY] - [I_aX,I_cY]) = 0. 
$$
Rewriting this as 
$$
\theta_b([X,I_cY] + [I_bI_cX,I_bY]) = - \theta_c([X,I_bY] + [I_cI_bX,I_cY]) 
$$
and using \eqref{integrability_condition0}, we conclude 
\begin{eqnarray}\label{intmodQ}
\theta_b([X,I_cY] + [I_cX,Y]) = - \theta_c([X,I_bY] + [I_bX,Y]), 
\end{eqnarray}
which also holds when $b=c$.
Therefore, 
$$
\theta_{\bf v}([ X,I_{\bf v}Y ] + [I_{\bf v}X, Y ])
= 
\sum_{b, c}{v_b}{v_c} \theta_b([X,{I_c}Y] + [{I_c}X,Y]) = 0. 
$$
Note that the terms involving the devivatives of $v_c$ 
disappear, since $\theta_b$ vanishes on $Q$. 
This proves \eqref{1stint}. 

Next we prove \eqref{2ndint} by showing that 
$$
\theta_a (I_{\mathbf v} ( [X,I_{\bf v}Y] + [ I_{\bf v}X,Y ] ) + [ X,Y ] - [ I_{\bf v}X,I_{\bf v}Y ] ) = 0
$$
for each $a$. 
The left-hand side is computed as
\begin{eqnarray*} 
&& \sum_{b,c} v_b v_c\, \theta_a ( \widetilde{I_b} ( [X,I_cY] + [ I_cX,Y ] ) - [I_bX, I_cY] ) + \theta_a([X,Y])\\ 
&=& \sum_b {v_b}^2 \theta_a ( I_b ( [X,I_bY] + [ I_bX,Y ] ) + [X,Y] - [I_bX, I_bY] ) \\ 
&& + \sum_{b\neq a} v_b v_a\, \theta_a ( \widetilde{I_b} ( [X,I_aY] + [ I_aX,Y ] ) 
+ \widetilde{I_a} ( [X,I_bY] + [ I_bX,Y ] ) \\ 
&& \phantom{+ \sum_{b\neq a} v_b v_a\, \theta_a (}
- [I_bX, I_aY] - [I_aX, I_bY] ) \\ 
&& + v_b v_c\, \theta_a ( \widetilde{I_b} ( [X,I_cY] + [ I_cX,Y ] ) + \widetilde{I_c} ( [X,I_bY] + [ I_bX,Y ] ) ) \\ 
&& \phantom{+ v_b v_c\, \theta_a (} - [I_bX, I_cY] - [I_cX, I_bY] ). 
\end{eqnarray*} 
Note that the sum over $b,c$ in the left-hand side is divided into four parts: 
the sums over $b=c$, $b\neq c=a$, $c\neq b=a$, and $b\neq c\neq a\neq b$. 
The second and third sums are grouped into the second sum of the right-hand side, 
and the fourth sum into the last term, in which  
$b,c$ are so that $(a,b,c)$ is a cyclic permutation of $(1,2,3)$.  
The first sum of the right-hand side vanishes by \eqref{integrability_condition}; 
the last term by \eqref{theta_ai_1}, 
\eqref{integrability_condition0}. 
The second sum also vanishes, since
\begin{eqnarray*} 
&& \theta_a ( \widetilde{I_b} ( [X,I_aY] + [ I_aX,Y ] ) = \theta_c ( [X,I_aY] + [ I_aX,Y ] ) \\ 
&=& - \theta_a ( [X,I_cY] + [ I_cX,Y ] ) = \theta_a ( [I_aX,I_bY] + [ I_bX,I_aY ] )
\end{eqnarray*} 
by \eqref{theta_ai_1}, \eqref{intmodQ} and \eqref{integrability_condition0}. 
This completes the proof of Proposition \ref{familyintegrable}.
\end{proof}

\subsection{Levi form}
\label{ap_Levi}

In this subsection, by applying Proposition \ref{familyintegrable}, we prove 
Propoisition \ref{prepLevi}, asserting that 
the Levi form on a quaternionic CR manifold is well-defined. 

\begin{Lemma}\label{replacing}
Let ${\mathbf u}, {\mathbf v}$ be mutually orthogonal unit vectors in $\R^3$ and $X, Y \in Q$. 
Then $d\theta_{\mathbf u}(I_{\mathbf v}X,I_{\mathbf v}Y)$ is independent of the choice of 
${\mathbf v}$ orthogonal to ${\mathbf u}$.
\end{Lemma}

\begin{proof}
Let ${\mathbf v}'$ be another unit vector orthogonal to ${\mathbf u}$.
Then ${\mathbf v}'$ can be expressed as
${\mathbf v}'={\lambda}{\mathbf v} + {\mu}{\mathbf u}\times {\mathbf v}$ 
with ${\lambda^2} + {\mu^2} = 1$,
and we have 
\begin{eqnarray*}
d\theta_{\mathbf u}(I_{{\mathbf v}'}X,I_{{\mathbf v}'}Y)
&=&
{\lambda^2}
d\theta_{\mathbf u}(I_{\mathbf v}X,I_{\mathbf v}Y)
+ {\mu^2}
d\theta_{\mathbf u}(I_{\mathbf u} I_{\mathbf v} X, I_{\mathbf u} I_{\mathbf v} Y) \\
&&
+ {\lambda}{\mu}
[ d\theta_{\mathbf u}( I_{\mathbf v}X,I_{\mathbf u} I_{\mathbf v} Y)
+ d\theta_{\mathbf u}(I_{\mathbf u} I_{\mathbf v} X,I_{\mathbf v}Y) ]. 
\end{eqnarray*}
Since $d\theta_{\mathbf u}$, restricted to $Q$, is $I_{\mathbf u}$-invariant, 
the right-hand side is equal to $d\theta_{\mathbf u}(I_{\mathbf v}X,I_{\mathbf v}Y)$.
\end{proof}

\noindent
{\em Proof of Proposition \ref{prepLevi}.}\quad  
Let $\theta_{U} = (\theta_a)$ and $\theta_{U'} = (\theta_a')$. 
Then $\theta_a' = \sum_{p=1}^3 s_{ap} \theta_p$ for an $SO(3)$-valued function $(s_{ap})$. 
We must show that 
$$
d\theta_1'(X, I_1'Y)+ d\theta_1'(I_2'X, I_3'Y) = d\theta_1(X, I_1Y) + d\theta_1(I_2X, I_3Y) 
$$
for all $X,Y\in Q$. 
Since $d\theta_1' = \sum_{p=1}^3 ( s_{1p} d\theta_p + ds_{1p} \wedge \theta_p )$ 
and $\theta_p$'s vanish on $Q$, we may assume that $s_{ap}$'s are constants. 
Therefore, it suffices to verify that 
\begin{eqnarray}\label{Levi_on_qua}
d\theta_{\mathbf u}(X,I_{\mathbf u}Y)
+ d\theta_{\mathbf u}(I_{\mathbf v}X, I_{{\mathbf u}\times{\mathbf v}}Y)
= d\theta_1(X,I_1Y)+d\theta_1(I_2X,I_3Y) 
\end{eqnarray} 
for all $X, Y \in Q$, where ${\mathbf u}, {\mathbf v}$ are any mutually orthogonal unit vectors in $\R^3$. 

Note that by Lemma \ref{replacing}, the left-hand side of \eqref{Levi_on_qua} is independent of
$\mathbf{v}$ (orthogonal to $\mathbf{u}$). 
It is easy to verify that one can choose $\mathbf{v}$ orthogonal to $\mathbf{u}$, and $\mathbf{v}'$ 
orthogonal to $\mathbf{u}' = \mathbf{u} \times \mathbf{v}$, so that ${\mathbf{u}'} \times {\mathbf{v}'} = 
{\mathbf{e}}_1 = \mbox{}^t (1,0,0)$.

Using \eqref{2ndint} and
$\theta_{\mathbf u} \circ I_{\mathbf v} = \theta_{{\mathbf u}\times{\mathbf v}}$
(cf.~\eqref{theta_ai_1}), we can rewrite the left-hand side of \eqref{Levi_on_qua} as 
\begin{eqnarray}\label{above_comp}
d\theta_{\mathbf u}(X,I_{\mathbf u}Y) - d\theta_{\mathbf u}(I_{\mathbf v}X,I_{\mathbf v}(I_{\mathbf u}Y)) 
&=& - d\theta_{{\mathbf u}\times{\mathbf v}} (X,I_{\mathbf v}(I_{\mathbf u}Y)) 
- d\theta_{{\mathbf u}\times{\mathbf v}}(I_{\mathbf v}X,I_{\mathbf u}Y) \nonumber \\
&=& d\theta_{{\mathbf u}'}(X,I_{{\mathbf u}'}Y) - d\theta_{{\mathbf u}'}(I_{\mathbf v}X,I_{\mathbf v}(I_{{\mathbf u}'}Y)) \\
&=& d\theta_{{\mathbf u}'}(X,I_{{\mathbf u}'}Y)
- d\theta_{{\mathbf u}'}(I_{{\mathbf v}'}X,I_{{\mathbf v}'}(I_{{\mathbf u}'}Y)). \nonumber
\end{eqnarray}
%
%
For the last equality, we have used Lemma \ref{replacing} again to replace ${\mathbf v}$ by ${\mathbf v}'$. 
Computing similarly as in \eqref{above_comp} while noting that ${\mathbf v}'$ and ${\mathbf{e}}_2 = \mbox{}^t (0,1,0)$ 
are both orthogonal 
to ${\mathbf{e}}_1$, we can verify that the last expression is equal to 
the right-hand side of \eqref{Levi_on_qua}. 
\hfill $\square$

\end{document}